\documentclass[a4paper, 11pt]{amsart}
\usepackage{graphicx, euscript, enumitem,kantlipsum}
\usepackage{amssymb}
\usepackage{amsmath}
\usepackage{amsthm}
\usepackage{amscd}
\usepackage{mathbbol}
\usepackage[all,2cell]{xy}

\usepackage[pagebackref,colorlinks]{hyperref}

\usepackage{tikz-cd}
\usetikzlibrary{arrows,backgrounds,shadows,decorations.markings,matrix}

\usepackage{geometry}
\UseAllTwocells \SilentMatrices
\newtheorem{thm}{Theorem}[section]

\newtheorem{cor}[thm]{Corollary}
\newtheorem{lem}[thm]{Lemma}

\newtheorem{thmy}{Theorem}

\newtheorem{prop}[thm]{Proposition}
\theoremstyle{definition}
\newtheorem{defn}[thm]{Definition}
\theoremstyle{remark}
\newtheorem{rem}[thm]{\bf Remark}
\numberwithin{equation}{section}

\newcommand{\Hom}{\mathrm{Hom}}
\newcommand{\sg}{\mathrm{sg}}

\makeatletter
\newcommand{\smallotimes}{\mathbin{\mathpalette\make@small\otimes}}

\newcommand{\make@small}[2]{%
  \vcenter{\hbox{%
    $\m@th\ifx#1\displaystyle\scriptstyle\else\ifx#1\textstyle\scriptstyle
     \else\scriptscriptstyle\fi\fi#2$%
  }}%
}
\makeatother

\begin{document}

\title{The singular Yoneda category and the stabilization functor}

\author{Xiao-Wu Chen and Zhengfang Wang}


\date{\today}

\begin{abstract}
For a noetherian ring $\Lambda$, the stabilization functor in the sense of Krause yields an embedding of the singularity category of $\Lambda$ into the homotopy category of acyclic complexes of injective $\Lambda$-modules. When $\Lambda$ contains a semisimple artinian subring $E$, we give an explicit description of the stabilization functor using the Hom complexes in the $E$-relative singular Yoneda dg category of $\Lambda$.
\end{abstract}

\maketitle


\section{Introduction}

Let $\Lambda$ be a  left noetherian ring. Denote by $\Lambda\mbox{-mod}$ the abelian category of finitely generated $\Lambda$-modules and by $\mathbf{D}^b(\Lambda\mbox{-mod})$ its bounded derived category. Following \cite{Buc, Orl},  the {\it singularity category} $\mathbf{D}_{\rm sg}(\Lambda)$ of $\Lambda$ is defined to be the Verdier quotient of $\mathbf{D}^b(\Lambda\mbox{-mod})$ modulo the full subcategory formed by perfect complexes. The singularity category $\mathbf{D}_{\rm sg}(\Lambda)$ measures the homological singularity of $\Lambda$ in the following sense: $\mathbf{D}_{\rm sg}(\Lambda)$ vanishes if and only if every finitely generated $\Lambda$-module has finite projective dimension. The singularity categories of certain hypersurfaces are related to categories of  B-branes in  Landau-Ginzburg models \cite{Orl}.

Denote by $\mathbf{K}(\Lambda\mbox{-Inj})$  the homotopy category of complexes of arbitrary injective $\Lambda$-modules and by $\mathbf{K}_{\rm ac}(\Lambda\mbox{-Inj})$ its full subcategory formed by acyclic complexes. It is shown in \cite{Kra} that $\mathbf{K}_{\rm ac}(\Lambda\mbox{-Inj})$ is a \emph{compactly generated completion} of $\mathbf{D}_{\rm sg}(\Lambda)$. More precisely, $\mathbf{K}_{\rm ac}(\Lambda\mbox{-Inj})$ is compactly generated and its full subcategory $\mathbf{K}_{\rm ac}(\Lambda\mbox{-Inj})^c$ of compact objects is triangle equivalent to $\mathbf{D}_{\rm sg}(\Lambda)$, up to direct summands.

We mention that such a concrete completion  makes it possible to apply the rich theory of compactly generated triangulated categories \cite{Neeman, Kel94} to  singularity categories. However, the relevant functor from $\mathbf{D}_{\rm sg}(\Lambda)$ to $\mathbf{K}_{\rm ac}(\Lambda\mbox{-Inj})$  is  neither explicit nor trivial  as explained below.

The following recollement \cite{BBD} among  compactly generated triangulated categories is established in \cite{Kra}.
\begin{equation}\label{rec:Kra-intr}
\xymatrix{
\mathbf{K}_{\rm ac}(\Lambda\mbox{-Inj})  \ar[rr]^-{\rm inc} &&  \ar@/_1.5pc/[ll]|{\bar{\bf a}}  \ar@/^1.5pc/[ll] \mathbf{K}(\Lambda\mbox{-Inj})  \ar[rr]^{\rm can} && \mathbf{D}(\Lambda\mbox{-Mod}) \ar@/_1.5pc/[ll]|{\bar{\bf p}}  \ar@/^1.5pc/[ll]|{\bf i}
}
\end{equation}
Here,  $\mathbf{D}(\Lambda\mbox{-Mod})$ is the unbounded derived category of complexes of $\Lambda$-modules, ``inc" is the inclusion functor and ``can" is the canonical functor. Following \cite{Kra}, the  composite functor
$$\mathbb{S}=\bar{\bf a}{\bf i}\colon \mathbf{D}(\Lambda\mbox{-}{\rm Mod})\longrightarrow \mathbf{K}_{\rm ac}(\Lambda\mbox{-}{\rm Inj})$$
is called the \emph{stabilization functor}. As $\mathbb{S}$ vanishes on perfect complexes, its restriction to $\mathbf{D}^b(\Lambda\mbox{-mod})$ induces a well-defined functor
\begin{align}\label{equ:ff}
\mathbf{D}_{\rm sg}(\Lambda)\longrightarrow \mathbf{K}_{\rm ac}(\Lambda\mbox{-Inj}).
\end{align}
The resulting functor yields an equivalence up to direct summands between $\mathbf{D}_{\rm sg}(\Lambda)$ and $\mathbf{K}_{\rm ac}(\Lambda\mbox{-Inj})^c$.

The stabilization functor $\mathbb{S}$ is crucial in the recollement (\ref{rec:Kra-intr}), as it is a triangulated analogue of the gluing functors in the dg  setting \cite{KL}  and $\infty$-categorical setting \cite{Lur, DJW}. By \cite{CL},  the comma category of $\mathbb{S}$  yields the middle term $\mathbf{K}(\Lambda\mbox{-Inj})$ up to an explicit equivalence. If the ring $\Lambda$ is \emph{Gorenstein}, that is, $\Lambda$ is two-sided noetherian such that it has finite selfinjective dimension on each side,  applying $\mathbb{S}$ to any $\Lambda$-module yields a complete injective resolution of the module; see \cite{Kra}. In other words, $\mathbb{S}$ provides a functorial construction for complete injective resolutions.

It is well known that the functor ${\bf i}$ in (\ref{rec:Kra-intr}) assigns to any complex its dg-injective resolution. However, the functor $\bar{\bf a}$ is not well understood. Consequently, the stabilization functor $\mathbb{S}=\bar{\bf a}{\bf i}$ is mysterious in a certain sense.

The goal of this work is to describe the stabilization functor $\mathbb{S}$ explicitly.

We will assume that $\Lambda$ contains a subring $E$ which is  semisimple artinian, for example a field. Our tool is the \emph{$E$-relative singular Yoneda dg category} $\mathcal{SY}=\mathcal{SY}_{\Lambda/E}$ introduced in \cite{CW}. The objects of $\mathcal{SY}$ are just complexes of $\Lambda$-modules, and for bounded complexes $X, Y$ of finitely generated modules, the Hom complex $\mathcal{SY}(X, Y)$ computes the shifted Hom groups in the singularity category $\mathbf{D}_{\rm sg}(\Lambda)$. In other words, $\mathcal{SY}$ contains a dg enhancement \cite{BK} of the singularity category. We mention that each Hom complex in  $\mathcal{SY}$ is defined as an explicit colimit; see Section \ref{sec:SY}.

The main results, proved in Sections~\ref{sec:SY} and \ref{sec:8},  are summarized as follows.  We  view $\Lambda$ as a stalk complex concentrated in degree zero, and thus an object in $\mathcal{SY}$.
\vskip 5pt

\noindent {\bf Theorem.} \;\emph{Let $\Lambda$ be a left noetherian ring and $E\subseteq \Lambda$ a semisimple artinian subring. Denote by $\mathcal{SY}$ the $E$-relative singular Yoneda dg category of $\Lambda$. Then the following hold.
\begin{enumerate}
\item $\mathcal{SY}(\Lambda, -)\colon \mathbf{D}(\Lambda\mbox{-}{\rm Mod})\rightarrow \mathbf{K}_{\rm ac}(\Lambda\mbox{-}{\rm Inj})$ is a well-defined triangle functor.
\item There is a natural transformation $c\colon \mathbb{S}\rightarrow \mathcal{SY}(\Lambda, -)$ such that its restriction to the full subcategory $\mathbf{D}^{+}(\Lambda\mbox{-}{\rm Mod})$ consisting of cohomologically bounded below complexes  is a natural isomorphism. Moreover, if $\Lambda$ is Gorenstein, $c$ is a natural isomorphism.
\end{enumerate}}

There are two immediate consequences of these results. By the first half of (2), we infer that the functor (\ref{equ:ff}) is naturally isomorphic to the following functor
\[
\mathcal{SY}(\Lambda, -) \colon \mathbf{D}_{\rm sg}(\Lambda)\longrightarrow \mathbf{K}_{\rm ac}(\Lambda\mbox{-Inj}),
\]
which is induced by the restriction of $\mathcal{SY}(\Lambda, -)$ to $\mathbf{D}^b(\Lambda\mbox{-mod})$.
 If $\Lambda$ is Gorenstein, we combine the second half of (2) with \cite[Section~7]{Kra}, and infer that $\mathcal{SY}(\Lambda, M)$ provides  an explicit complete injective resolution for any $\Lambda$-module $M$.

In a certain sense, the whole paper is devoted to the proof of the above results.

The first key step is to describe the two functors  ${\bf i}$ and $\bar{\bf p}$ in \eqref{rec:Kra-intr} using the \emph{$E$-relative Yoneda dg category} $\mathcal{Y}=\mathcal{Y}_{\Lambda/E}$ of $\Lambda$, which is a natural dg enhancement of  the derived category using the bar resolution \cite{Kel94, CW}; see Propositions~\ref{prop:Y-i} and \ref{prop:bar-p}. The second one is to interpret both $\mathbb{S}$ and $\mathcal{SY}(\Lambda, -)$ as the mapping cones of explicit quasi-isomorphisms, so that it is possible to compare them; see Theorems~\ref{thm:SY} and \ref{thm:S}.  We mention that such an interpretation actually  lifts the stabilization functor  $\mathbb{S}$ to the dg level.

The paper is structured as follows. In Section~2, we recall the dg-projective and dg-injective resolutions of complexes using the bar resolution. In Section~3, we recall the Yoneda dg category $\mathcal{Y}$ in \cite{CW} and prove in Proposition~\ref{prop:Y-i} that its Hom complexes yield dg-injective resolutions of complexes. For later use, we study an explicit quasi-isomorphism (\ref{equ:epsilon}) in Section~4. In Section~5, we study noncommutative differential forms with values in complexes, and their compatibility with the truncated bar resolutions as shown by an explicit commutative diagram in $\mathcal{Y}$; see Proposition~\ref{prop:Omega-bar}.

In Section~6, we recall the singular Yoneda dg category $\mathcal{SY}$ in \cite{CW}. We prove in Theorem~\ref{thm:SY} that for any complex $X$ of $\Lambda$-modules,  the Hom complex $\mathcal{SY}(\Lambda, X)$ is homotopy equivalent to the mapping cone of an explicit quasi-isomorphism $\vartheta_X$ in (\ref{equ:vartheta}). In Section~7, we describe the functor $\bar{\bf p}$ in \eqref{rec:Kra-intr} in Proposition~\ref{prop:bar-p}. We prove in Theorem~\ref{thm:S} that $\mathbb{S}(X)$ is homotopy equivalent to the mapping cone of an explicit quasi-isomorphism $\kappa_X$ in (\ref{equ:kappa}). In the final section, we compare $\mathcal{SY}(\Lambda, X)$ and $\mathbb{S}(X)$ in Theorem~\ref{thm:comparison}. As an application of the comparison, we lift a result in \cite{Kra} to the dg level.

Throughout the paper, we fix a semisimple artinian ring $E$. The unadorned Hom and tensor are over $E$. We will always assume that $\Lambda$ is a ring containing $E$ as a subring with the same unit. By $\Lambda$-modules, we mean left $\Lambda$-modules, and by complexes, we mean cochain complexes. We will abbreviate ``differential graded" as dg.

\section{Resolutions of complexes via the bar resolution}\label{section:singularyonedacomplex}

In this section, we recall basic facts on dg-projective resolutions and dg-injective resolutions of complexes via the bar resolution.

 We denote by $\Lambda\mbox{-Mod}$ the category of $\Lambda$-modules, and by $C(\Lambda\mbox{-Mod})$ the category of complexes of $\Lambda$-modules. Denote by $\mathbf{K}(\Lambda\mbox{-Mod})$ the homotopy category of complexes of $\Lambda$-modules and by  $\mathbf{D}(\Lambda\mbox{-Mod})$  the derived category. We will always view a module as a stalk complex concentrated in degree zero.

 A complex of $\Lambda$-modules is usually denoted by $X=(X^n, d_X^n)_{n\in \mathbb{Z}}$, where $d_X^n$ is often written simply as $d_X$. We will use $\Sigma$ to denote the suspension functor for cochain complexes. To be more precise, the suspended complex $\Sigma(X)$ is described by $\Sigma(X)^n=X^{n+1}$ and $d_{\Sigma(X)}^n=-d_X^{n+1}$.

For a cochain map $f=(f^n)_{n\in \mathbb{Z}}\colon X\rightarrow Y$ between complexes, the mapping cone ${\rm Cone}(f)$ is a complex described as follows:
$${\rm Cone}(f)^n=Y^n\oplus X^{n+1}, \quad d_{{\rm Cone}(f)}^n=\begin{pmatrix}d_Y^n & f^{n+1}\\
0 & -d_X^{n+1} \end{pmatrix}.$$
We have the following standard exact triangle in $\mathbf{K}(\Lambda\mbox{-Mod})$:
\begin{align}\label{tri:standard}
X\stackrel{f}\longrightarrow Y \stackrel{\binom{\mathbf{1}}{0}} \longrightarrow {\rm Cone}(f)\stackrel{(0\; \mathbf{1})}\longrightarrow \Sigma(X).
\end{align}
Here, we use $\mathbf{1}$ to denote the identity endomorphism.
\subsection{The bar resolution}\label{subs:2.1}

As $E$ is  a subring of $\Lambda$,  $\Lambda$ is naturally an $E$-$E$-bimodule.  Denote by $\bar \Lambda=\Lambda/E$ the quotient $E$-$E$-bimodule; as above, it is always viewed as a stalk complex concentrated in degree zero. Denote by $s\bar{\Lambda}$ the $1$-shifted complex, which is a stalk complex concentrated in degree $-1$. For any $a\in \Lambda$, the corresponding element in $s\bar{\Lambda}$ is written as  $s\bar{a}$. Denote by
$$T(s\bar{\Lambda})=E\oplus s\bar{\Lambda}\oplus (s\bar{\Lambda})^{\otimes 2}\oplus \cdots$$
the tensor ring of the $E$-$E$-bimodule $s\bar{\Lambda}$. A typical element in $(s\bar{\Lambda})^{\otimes n}$ is of the form $s\bar{a}_1\otimes \cdots \otimes s\bar{a}_n$, which is often abbreviated as $s{\bar{a}}_{1,n}$. We observe that the degree of $s{\bar{a}}_{1,n}$ is  $-n$.

The {\it normalised $E$-relative bar resolution} $\mathbb{B}$ is a complex of $\Lambda$-$\Lambda$-bimodules
\[
 \mathbb{B}: = \Lambda \otimes T(s \bar \Lambda)\otimes \Lambda,
\]
whose differential is the external one $d_{\rm ex}$ given by
\begin{equation}\label{equ:external}
\begin{split}
 d_{\rm ex}(a_0 \smallotimes  s  {\bar a_{1, n}}\smallotimes  a_{n+1})={}&a_0a_1\smallotimes  s  { \bar a_{2, n}} \smallotimes  a_{n+1}  + (-1)^{n} a_0\smallotimes  s  {\bar a_{1, n-1}} \smallotimes  a_na_{n+1}\\
 & +\sum_{i=1}^{n-1}(-1)^{i}a_0 \smallotimes  s  { \bar a_{1, i-1}} \smallotimes  s  {\overline{a_{i}a_{i+1}}} \smallotimes  s  {\bar a_{i+2, n}} \smallotimes  a_{n+1}.
\end{split}
\end{equation}
Here, the expressions $s\bar{a}_{1, 0}\smallotimes$ for $i=1$, and $s\bar{a}_{n+1, n}\smallotimes$ for $i=n-1$, are understood as the empty word.

For each $p\geq 0$, we consider the following subcomplex of $\mathbb{B}$:
$$\mathbb{B}_{<p}=\bigoplus_{0\leq n < p} \Lambda \otimes (s\bar{\Lambda})^{\otimes n}\otimes \Lambda.$$
Here, we understand $\mathbb{B}_{<0}$ as the zero complex, and $\mathbb{B}_{<1}$ as $\Lambda \otimes \Lambda=\Lambda \otimes E\otimes \Lambda$. The corresponding  quotient complex $\mathbb{B}/\mathbb{B}_{<p}$ will be denoted by $\mathbb{B}_{\geq p}$.

Composing the projection $\mathbb{B}\rightarrow \Lambda\otimes \Lambda$ with the multiplication map $\Lambda\otimes \Lambda \rightarrow \Lambda$, we obtain the following natural map
$$\varepsilon\colon \mathbb{B}\longrightarrow \Lambda.$$
It  is a quasi-isomorphism between complexes of $\Lambda$-$\Lambda$-bimodules. Here, as above $\Lambda$ is viewed a stalk complex concentrated in degree zero. As in (\ref{tri:standard}), we have the following standard  exact triangle of complexes of $\Lambda$-$\Lambda$-bimodules
\begin{align}\label{tri:bar}
\mathbb{B}\stackrel{\varepsilon }\longrightarrow \Lambda \longrightarrow {\rm Cone}(\varepsilon)\longrightarrow \Sigma(\mathbb{B}).
\end{align}
We refer to \cite[Section~8.6]{Wei} and \cite[Subsection~6.6]{Kel94} for more details on the bar resolution.

\subsection{Resolutions of complexes}\label{subs:2.2}

For two complexes $X=(X^n, d^n_X)_{n\in \mathbb{Z}}$ and $Y=(Y^n, d^n_Y)_{n\in \mathbb{Z}}$ of $\Lambda$-modules, the \emph{Hom-complex} ${\rm Hom}_\Lambda(X, Y)$ is a complex of abelian groups defined as follows: its $p$-th component ${\rm Hom}_\Lambda(X, Y)^p$ consists of graded maps $f\colon X\rightarrow Y$ of graded $\Lambda$-modules that have degree $p$, namely, $f(X^n)\subseteq Y^{p+n}$ for each $n \in \mathbb{Z}$; its differential is defined such that
$$d(f)=d_Y\circ f-(-1)^{|f|}f\circ d_X.$$
The following well-known isomorphisms
\begin{align}\label{iso:keyfor}
H^n {\rm Hom}_\Lambda(X, Y)\simeq {\rm Hom}_{\mathbf{K}(\Lambda\mbox{-}{\rm Mod})}(X, \Sigma^n(Y)),\quad  \forall n\in \mathbb{Z},
\end{align}
will be used often.

Recall that a  complex $P$ of $\Lambda$-modules is \emph{$\mathbf{K}$-projective} if the Hom-complex ${\rm Hom}_\Lambda (P, Z)$ is acyclic for any acyclic complex $Z$ of $\Lambda$-modules; a complex $P$
is \emph{dg-projective} provided that it is $\mathbf{K}$-projective and each component $P^n$ is a projective $\Lambda$-module. Dually, a complex $I$ of $\Lambda$-modules is \emph{$\mathbf{K}$-injective} if the Hom-complex ${\rm Hom}_\Lambda (Z, I)$ is acyclic for any acyclic complex $Z$ of $\Lambda$-modules. A $\mathbf{K}$-injective complex $I$ is called \emph{dg-injective} if in addition each component $I^n$ is an injective module.

We refer to \cite[Section~1.4]{BBD} for details on recollements. We mention its analogues in the dg setting \cite[Section~4]{KL} and  $\infty$-categorical setting; see \cite[Appendix~A.8]{Lur} and \cite[Example~1.4]{DJW}. In drawing an adjoint pair, we always put the left adjoint functor in  the upper position.

Denote by $\mathbf{K}_{\rm ac}(\Lambda\mbox{-Mod})$ the homotopy category of acyclic complexes of $\Lambda$-modules. The following recollement is well known.
\begin{equation}\label{rec:cal}
\xymatrix{
\mathbf{K}_{\rm ac}(\Lambda\mbox{-Mod})  \ar[rr]^-{\rm inc} &&  \ar@/_1.5pc/[ll]|{\bf a}  \ar@/^1.5pc/[ll]|{{\bf a}'} \mathbf{K}(\Lambda\mbox{-Mod})  \ar[rr]^{\rm can} && \mathbf{D}(\Lambda\mbox{-Mod}) \ar@/_1.5pc/[ll]|{\bf p}  \ar@/^1.5pc/[ll]|{\bf i}
}
\end{equation}
Here, ``inc'' denotes the inclusion and ``can" denotes the quotient functor. For each $X$, we have a unique exact triangle in $\mathbf{K}(\Lambda\mbox{-Mod})$:
\begin{align}\label{tri:p}
{\bf p}(X)\longrightarrow X\longrightarrow {\bf a}(X)\longrightarrow \Sigma {\bf p}(X)
\end{align}
such that ${\bf p}(X)$ is dg-projective  and ${\bf a}(X)$ is acyclic; consequently, the cochain map ${\bf p}(X)\rightarrow X$ is a quasi-isomorphism. Dually, we have a unique exact triangle in $\mathbf{K}(\Lambda\mbox{-Mod})$:
\begin{align}\label{tri:i}{\bf a}'(X)\longrightarrow X\longrightarrow {\bf i}(X)\longrightarrow \Sigma {\bf a}'(X)
\end{align}
such that ${\bf i}(X)$ is dg-injective and ${\bf a}'(X)$ is acyclic; thus the cochain map $X\rightarrow {\bf i}(X)$ is a quasi-isomorphism.

We call ${\bf p}(X)$ and ${\bf i}(X)$ the \emph{dg-projective resolution} and \emph{dg-injective resolution} of $X$, respectively. Indeed, one should call the quasi-isomorphisms ${\bf p}(X)\rightarrow X$ and $X\rightarrow {\bf i}(X)$ the corresponding resolutions. For details, we refer to \cite[Section~3]{Kel94} and \cite[Section~4.3]{Kra22}.

 We mention that the recollement (\ref{rec:cal}) exists for any ring. In general, the four non-obvious functors are not explicitly given. In what follows, using the semisimple artinian subring $E$ and  the $E$-relative bar resolution $\mathbb{B}$, we describe these functors more explicitly.

 \begin{lem}\label{lem:projres}
Let $X$ be any complex of $\Lambda$-modules. Then as a complex of $\Lambda$-modules, $\mathbb{B}\otimes_\Lambda X$ is dg-projective, and the following map
$$\varepsilon \otimes_\Lambda {\rm Id}_X\colon \mathbb{B}\otimes_\Lambda X\longrightarrow \Lambda\otimes_\Lambda X=X$$
is a quasi-isomorphism. Consequently, we have isomorphisms ${\bf p}\simeq \mathbb{B}\otimes_\Lambda-$ and ${\bf a}\simeq {\rm Cone}(\varepsilon)\otimes_\Lambda-$ of functors.
\end{lem}

\begin{proof}
We observe that $\mathbb{B}\otimes_\Lambda X$ is the union of the following ascending chain of subcomplexes
$$\mathbb{B}_{\leq 0}\otimes_\Lambda X\subseteq \mathbb{B}_{\leq 1}\otimes_\Lambda X\subseteq \mathbb{B}_{\leq 2}\otimes_\Lambda X\subseteq \cdots.$$
These inclusions are componentwise split, and the factors are isomorphic to
$$\mathbb{B}^{-p}\otimes_\Lambda X\simeq \Lambda\otimes (s\bar{\Lambda})^{\otimes p}\otimes X.$$
As $E$ is semisimple artinian, we infer that these factors are isomorphic to direct summands of coproducts of $\Sigma^m(\Lambda)$. It follows that $\mathbb{B}\otimes_\Lambda X$ satisfies the property (P) in \cite[Section~3]{Kel94}, and thus is dg-projective.

As $\varepsilon$ is a homotopy equivalence between the underlying complexes of right $\Lambda$-modules, the map $\varepsilon\otimes_\Lambda{\rm Id}_X$ is a homotopy equivalence of complexes of abelian groups. In particular, it is a quasi-isomorphism. Now, applying $-\otimes_\Lambda X$ to the standard triangle (\ref{tri:bar}), we obtain an exact triangle, that is isomorphic to (\ref{tri:p}). Then we infer the required isomorphisms of functors.
\end{proof}

The following dual lemma seems to be less well known.

\begin{lem}\label{lem:injres}
Let $X$ be any complex of $\Lambda$-modules. Then as a complex of $\Lambda$-modules, ${\rm Hom}_\Lambda(\mathbb{B}, X)$ is dg-injective, and the following map
$${\rm Hom}_\Lambda(\varepsilon, X) \colon X={\rm Hom}_\Lambda(\Lambda, X) \longrightarrow {\rm Hom}_\Lambda(\mathbb{B}, X)$$
is a quasi-isomorphism.
 Consequently, we have isomorphisms ${\bf i}\simeq {\rm Hom}_\Lambda(\mathbb{B}, -)$ and ${\bf a}'\simeq {\rm Hom}_\Lambda({\rm Cone}(\varepsilon), -)$ of functors.
\end{lem}

\begin{proof}
The $p$-th component of ${\rm Hom}_\Lambda(\mathbb{B}, X)$ is an infinite product $\prod_{n\geq 0} {\rm Hom}_\Lambda(\mathbb{B}^{-n}, X^{p-n})$. We observe the following canonical isomorphism
$$ {\rm Hom}_\Lambda(\mathbb{B}^{-n}, X^{p-n})\simeq {\rm Hom}(\bar{\Lambda}^{\otimes n}\otimes \Lambda, X^{p-n}).$$
As $\bar{\Lambda}^{\otimes n}\otimes \Lambda$ is a projective right $\Lambda$-module, we infer by the Hom-tensor adjunction that ${\rm Hom}(\bar{\Lambda}^{\otimes n}\otimes \Lambda, X^{p-n})$ is an injective $\Lambda$-module. This proves that each component of  ${\rm Hom}_\Lambda(\mathbb{B}, X)$  is injective.

Take any acyclic complex $Z$ of $\Lambda$-module. We have the following canonical isomorphism of complexes
$${\rm Hom}_\Lambda (Z, {\rm Hom}_\Lambda(\mathbb{B}, X)) \simeq {\rm Hom}_\Lambda(\mathbb{B}\otimes_\Lambda Z, X).$$
By Lemma~\ref{lem:projres}, $\mathbb{B}\otimes_\Lambda Z$ is dg-projective and acyclic, as it is quasi-isomorphic to $Z$. Then it is contractible. It follows that the Hom-complex ${\rm Hom}_\Lambda(\mathbb{B}\otimes_\Lambda Z, X)$ is acyclic; compare (\ref{iso:keyfor}).  Therefore, ${\rm Hom}_\Lambda (Z, {\rm Hom}_\Lambda(\mathbb{B}, X))$ is also acyclic. In summary, we have proved that ${\rm Hom}_\Lambda(\mathbb{B}, X)$ is $\mathbf{K}$-injective and thus  dg-injective.

As $\varepsilon$ is a homotopy equivalence between the underlying complexes of left $\Lambda$-modules, it follows that ${\rm Hom}_\Lambda(\varepsilon, X)$ is a homotopy equivalence between the Hom-complexes of abelian groups. We infer that it is a quasi-isomorphism.

The required isomorphisms of functors follow by applying ${\rm Hom}_\Lambda(-, X)$ to (\ref{tri:bar}) and comparing the resulting triangle with (\ref{tri:i}).
\end{proof}

\begin{rem}\label{rem:homo}
 Assume that $X$ is a complex of $\Lambda$-$\Lambda$-bimodules. Then  by the same reasoning in the third paragraph of the above proof, the quasi-isomorphism ${\rm Hom}_\Lambda(\varepsilon, X)$ is even a homotopy equivalence between the complexes of right $\Lambda$-modules.
\end{rem}

In summary, we infer from Lemmas~\ref{lem:projres} and \ref{lem:injres} that the recollement \eqref{rec:cal} may be rewritten as follows.
\begin{equation}\label{rec:calrevised}
\xymatrix@C=3.5pc{
\mathbf{K}_{\rm ac}(\Lambda\mbox{-Mod})  \ar[rr]^-{\rm inc} &&  \ar@/_1.5pc/[ll]|{\mathbf{a}={\rm Cone}(\varepsilon)\otimes_\Lambda-}  \ar@/^1.5pc/[ll]|{\mathbf{a}'= {\rm Hom}_\Lambda({\rm Cone}(\varepsilon), -)} \mathbf{K}(\Lambda\mbox{-Mod})  \ar[rr]^{\rm can} && \mathbf{D}(\Lambda\mbox{-Mod}) \ar@/_1.5pc/[ll]|{\mathbf{p}=\mathbb{B}\otimes_\Lambda-}  \ar@/^1.5pc/[ll]|{\mathbf{i}={\rm Hom}_\Lambda(\mathbb{B}, -)}
}
\end{equation}

\section{The Yoneda dg category and dg-injective resolutions}
We will describe  dg-injective resolutions of complexes via the Yoneda dg category introduced in \cite{CW}.

\subsection{Preliminaries on dg categories}
We will recall two basic results on dg categories. The main references for dg categories are \cite{Kel94, Kel06}.

Let $\mathcal{C}$ be a dg category.  For two objects $A$ and $B$, the Hom set is usually denoted by $\mathcal{C}(A, B)$, which is a complex of abelian groups. A homogeneous morphism $f\colon A\rightarrow B$ is \emph{closed}, if $d_\mathcal{C}(f)=0$.

\begin{lem}\label{lem:contra}
For an object $A$ in a dg category $\mathcal{C}$, the following statements are equivalent:
\begin{enumerate}
\item[(1)] $H^0(\mathcal{C}(A, A))=0$;
\item[(2)] ${\rm Id}_A\in \mathcal{C}(A, A)$ is a coboundary;
\item[(3)]  for each object $X$, the Hom complex $\mathcal{C}(A, X)$ is acyclic.
\end{enumerate}
\end{lem}

We will say that such an object $A$ is \emph{contractible} in $\mathcal{C}$. Thanks to (2), any dg functor sends contractible objects to contractible objects.

\begin{proof}
The implication ``$(3)\Rightarrow (1)$" is trivial. We observe that ${\rm Id}_A$ is always closed, that is, a cocycle in $\mathcal{C}(A, A)$. Then we infer ``$(1)\Rightarrow (2)$".

It remains to show ``$(2)\Rightarrow (3)$". We fix $u\in \mathcal{C}(A, A)$ of degree $-1$ satisfying $d_\mathcal{C}(u)={\rm Id}_A$. For any cocycle $f\in \mathcal{C}(A, X)$, using the graded Leibniz rule we have
$$d_\mathcal{C}(f\circ u)=d_\mathcal{C}(f)\circ u+(-1)^{|f|}f\circ d_\mathcal{C}(u)=(-1)^{|f|}f.$$
It implies that $f$ is a coboundary, as required.
\end{proof}

Let $\mathcal{C}$ be a dg category.  The \emph{homotopy category} $H^0(\mathcal{C})$ is a pre-additive category with the same objects as $\mathcal{C}$ such that $H^0(\mathcal{C})(A, B)=H^0(\mathcal{C}(A, B))$. A closed morphism $f\colon A\rightarrow B$ of degree zero is called a \emph{homotopy equivalence} in $\mathcal{C}$, if its image in $H^0(\mathcal{C})$ is an isomorphism. It is equivalent to the following condition: there is a closed morphism $g\colon B\rightarrow A$ of degree zero such that both $g\circ f-{\rm Id}_A$ and $f\circ g-{\rm Id}_B$ are coboundaries; such a morphism $g$ is called a \emph{homotopy inverse} of $f$.

 \begin{lem}\label{lem:homo-inv}
For a closed morphism $f\colon A\rightarrow B$ of degree zero in a dg category $\mathcal{C}$, the following statements are equivalent:
\begin{enumerate}
\item[(1)]  for any object $X$, the cochain map $\mathcal{C}(f, X)\colon \mathcal{C}(B, X)\rightarrow \mathcal{C}(A, X)$ induces an isomorphism between $H^0(\mathcal{C}(B, X))$ and $H^0(\mathcal{C}(A, X))$;
\item[(2)]  $f$ is a homotopy equivalence in $\mathcal{C}$;
\item[(3)]  for any object $X$, the cochain map $\mathcal{C}(f, X)\colon \mathcal{C}(B, X)\rightarrow \mathcal{C}(A, X)$ is an quasi-isomorphism.
\end{enumerate}
\end{lem}

\begin{proof}
Since $f$ is closed of degree zero, the map $\mathcal{C}(f, X)$ is indeed a cochain map. For ``$(1)\Rightarrow (2)$", we have that $H^0(\mathcal{C})(f, X)$ is an isomorphism for any object $X$. By Yoneda Lemma, we infer that $f$ is an isomorphism in $H^0(\mathcal{C})$.

The implication ``$(3)\Rightarrow (1)$" is trivial. It remains to show ``$(2)\Rightarrow (3)$". For this, we take a homotopy inverse $g$ of $f$. Then it is direct to verify that $\mathcal{C}(g, X)$ is homotopy inverse of $\mathcal{C}(f, X)$ in the category of complexes of abelian groups. In particular, $\mathcal{C}(f, X)$ is a quasi-isomorphism.
\end{proof}

The main example of a dg category is the dg category $C_{\rm dg}(\Lambda\mbox{-Mod})$ formed by complexes of $\Lambda$-modules: the Hom sets are given by the corresponding Hom-complexes; see Subsection~\ref{subs:2.2}. We observe that the homotopy category $H^0(C_{\rm dg}(\Lambda\mbox{-Mod}))$ coincides with the usual homotopy category $\mathbf{K}(\Lambda\mbox{-Mod})$. The contractible objects in $C_{\rm dg}(\Lambda\mbox{-Mod})$ are precisely the usual contractible complexes. Similarly, homotopy equivalences in $C_{\rm dg}(\Lambda\mbox{-Mod})$ are precisely the usual homotopy equivalences between complexes.

\subsection{The Yoneda dg category}

Following \cite[Section~7]{CW}, we define the \emph{$E$-relative Yoneda dg category} $\mathcal{Y}=\mathcal{Y}_{\Lambda/E}$ of $\Lambda$ as follows.  It has the same objects as $C_{\rm dg}(\Lambda\mbox{-Mod})$. For two complexes $X$ and $Y$ of $\Lambda$-modules, the underlying graded group of $\mathcal{Y}(X, Y)$ is given by an infinite product
$$\mathcal{Y}(X, Y)=\prod_{p\geq 0}{\rm Hom}((s\bar{\Lambda})^{\otimes p}\otimes X, Y).$$
Set
$$\mathcal{Y}_p(X, Y):={\rm Hom}((s\bar{\Lambda})^{\otimes p}\otimes X, Y).$$
 We observe that $\mathcal{Y}_0(X, Y)={\rm Hom}(X, Y)$. Elements in $\mathcal{Y}_p(X, Y)$ is said to have \emph{filtration-degree} $p$. As usual, for a graded map $f\in \mathcal{Y}(X, Y)$, we denote by $|f|$ its cohomological degree.   The differential $\delta$ of $\mathcal{Y}(X, Y)$ is determined by
$$\begin{pmatrix}\delta_{\rm in}\\
\delta_{\rm ex}\end{pmatrix}\colon \mathcal{Y}_p(X, Y)\longrightarrow \mathcal{Y}_p(X, Y)\oplus \mathcal{Y}_{p+1}(X, Y),$$
where the internal one is given by
$$\delta_{\rm in}(f)(s\bar{a}_{1,p}\otimes x)=d_Y (f(s\bar{a}_{1,p}\otimes x))-(-1)^{|f|+p}f(s\bar{a}_{1,p}\otimes d_X(x))$$
and the external one is given by
\begin{align*}
\delta_{\rm ex}(f)(s\bar{a}_{1,p+1}\otimes x) = & (-1)^{|f|+1}a_1 f(s\bar{a}_{2,p+1}\otimes x)+(-1)^{|f|+p}f(s\bar{a}_{1,p}\otimes a_{p+1}x)\\
&+\sum_{i=1}^{p}(-1)^{|f|+i+1} f(s\bar{a}_{1, i-1}\otimes s\overline{a_ia_{i+1}}\otimes s\bar{a}_{i+2, p+1}\otimes x).
\end{align*}
Here, as in Subsection~\ref{subs:2.1}, the expressions $s\bar{a}_{1, 0}\otimes$ for $i=1$, and $s\bar{a}_{p+2, p+1}\otimes$ for $i=p$, are understood as the empty word.

The composition $\odot$ of morphisms in $\mathcal{Y}$ is defined as follows: for $f\in \mathcal{Y}_p(X, Y)$ and $g\in \mathcal{Y}_q(Y, Z)$, their composition $g\odot f\in \mathcal{Y}_{p+q}(X, Z)$ is given such that
\begin{equation}\label{equation:yonedaproduct}
(g\odot f)(s\bar{a}_{1, p+q}\otimes x)=(-1)^{q\cdot|f|} g(s\bar{a}_{1, q}\otimes f(s\bar{a}_{q+1, p+q}\otimes x)).
\end{equation}
The identity endomorphism in $\mathcal{Y}(X, X)$ is given by the identity map ${\rm Id}_X\in \mathcal{Y}_0(X, X)$. We mention that $\mathcal{Y}$ is implicit in \cite[Subsection~6.6]{Kel94}.

By \cite[the proof of Lemma~7.1]{CW}, we have a canonical isomorphism of complexes
\begin{align}\label{iso:Y-B}
\alpha_{X, Y}\colon \mathcal{Y}(X, Y) \longrightarrow {\rm Hom}_\Lambda(\mathbb{B}\otimes_\Lambda X, Y), \quad f\mapsto \tilde{f}.
\end{align}
 The isomorphism sends $f\in \mathcal{Y}_p(X, Y)$ to $\tilde{f}\colon (\Lambda\otimes (s\bar{\Lambda})^{\otimes p}\otimes \Lambda) \otimes_\Lambda X\rightarrow Y $ given by
$$\tilde{f}((a\otimes s\bar{a}_{1, p} \otimes b)\otimes_\Lambda x)=a f(s\bar{a}_{1, p}\otimes bx).$$
By Lemma~\ref{lem:projres}, the Hom-complex  ${\rm Hom}_\Lambda(\mathbb{B}\otimes_\Lambda X, Y)$ computes the Hom groups in the derived category $\mathbf{D}(\Lambda\mbox{-Mod})$. Consequently,  we have isomorphisms
\begin{align}\label{iso:Y-D}
H^n(\mathcal{Y}(X, Y))\simeq {\rm Hom}_{\mathbf{D}(\Lambda\mbox{-}{\rm Mod})}(X, \Sigma^n(Y))
\end{align}
for all $n\in \mathbb{Z}$.

\begin{rem}\label{rem:c}
Assume that $X$ is a complex of $\Lambda$-$\Lambda$-bimodules. Then both $\mathcal{Y}(X, Y)$ and $ {\rm Hom}_\Lambda(\mathbb{B}\otimes_\Lambda X, Y)$ are  complexes of  $\Lambda$-modules. Then the isomorphism $\alpha_{X, Y}$ becomes an isomorphism of complexes of $\Lambda$-modules. Taking $X=\Lambda$, we infer that $\mathcal{Y}(\Lambda, Y)$ is a complex of $\Lambda$-modules; moreover, by applying $\alpha_{\Lambda, Y}$ and Lemma~\ref{lem:injres}, it is even dg-injective.
\end{rem}

The natural inclusion 
\begin{align}\label{align:naturalinclusion}
{\rm Hom}_\Lambda(X, Y)\subseteq \mathcal{Y}_0(X, Y)\subseteq \mathcal{Y}(X, Y)
\end{align} 
makes ${\rm Hom}_\Lambda(X, Y)$ a subcomplex of $\mathcal{Y}(X, Y)$. Therefore, we view $C_{\rm dg}(\Lambda\mbox{-Mod})$ as a non-full dg subcategory of $\mathcal{Y}$. In particular, cochain maps between complexes are viewed as morphisms in $\mathcal{Y}$, that have filtration-degree zero.

\begin{prop}\label{cor:Y}
Keep the notation as above. Then the following two statements hold.
\begin{enumerate}
\item  Any acyclic complex $X$ is a contractible object in $\mathcal{Y}$.
\item  Any quasi-isomorphism $f\colon X\rightarrow Y$ between complexes is a homotopy equivalence in $\mathcal{Y}$.
\end{enumerate}
\end{prop}

\begin{proof}
Recall that any acyclic complex is a zero object in $\mathbf{D}(\Lambda\mbox{-}{\rm Mod})$, and that any quasi-isomorphism between complexes becomes an isomorphism in $\mathbf{D}(\Lambda\mbox{-}{\rm Mod})$. Combining the isomorphism (\ref{iso:Y-D}) with  Lemma~\ref{lem:contra}(3), we infer (1). Similarly, using (\ref{iso:Y-D}) and Lemma~\ref{lem:homo-inv}(3), we infer (2).
\end{proof}

 Consequently, we have a dg functor
$$\mathcal{Y}(\Lambda, -)\colon C_{\rm dg}(\Lambda\mbox{-Mod})\longrightarrow C_{\rm dg}(\Lambda\mbox{-Mod}),$$
which induces a triangle endofunctor
$$\mathcal{Y}(\Lambda, -)\colon \mathbf{K}(\Lambda\mbox{-Mod})\longrightarrow \mathbf{K}(\Lambda\mbox{-Mod})$$
 between the homotopy categories. By the following lemma, we have an induced triangle functor
$$\mathcal{Y}(\Lambda, -)\colon \mathbf{D}(\Lambda\mbox{-Mod})\longrightarrow \mathbf{K}(\Lambda\mbox{-Mod}).$$

\begin{lem}\label{lem:Y-quasi-iso}
For any quasi-isomorphism $g\colon Y\rightarrow Y'$ of complexes of $\Lambda$-modules, we have that $\mathcal{Y}(\Lambda, g)\colon \mathcal{Y}(\Lambda, Y)\rightarrow \mathcal{Y}(\Lambda, Y')$ is an isomorphism in $\mathbf{K}(\Lambda\mbox{-}{\rm Mod})$.
\end{lem}

\begin{proof}
Using the isomorphisms $\alpha_{\Lambda, Y}$ and $\alpha_{\Lambda, Y'}$, it suffices to show that
$${\rm Hom}_\Lambda(\mathbb{B}\otimes_\Lambda \Lambda, g)={\rm Hom}_\Lambda(\mathbb{B}, g)$$
is a homotopy equivalence.  Recall from Lemma~\ref{lem:injres} the isomorphism ${\rm Hom}_\Lambda(\mathbb{B}, -)\simeq {\bf i}$. In particular, both functors send quasi-isomorphisms to homotopy equivalences. Then the required statement follows.
\end{proof}

\begin{rem}
One might prove the above lemma alternatively by using Proposition~\ref{cor:Y}(2) and the dual of Lemma~\ref{lem:homo-inv}(3).
\end{rem}

We describe the  dg-injective resolution functor ${\bf i}\colon  \mathbf{D}(\Lambda\mbox{-Mod})\rightarrow \mathbf{K}(\Lambda\mbox{-Mod})$  in the recollement (\ref{rec:calrevised}) via the Yoneda dg category.

\begin{prop}\label{prop:Y-i}
There is an isomorphism $\mathcal{Y}(\Lambda, -)\simeq {\bf i}$ of triangle functors.
\end{prop}

\begin{proof}
The isomorphisms $\alpha_{\Lambda, Y}$ imply that
$$\mathcal{Y}(\Lambda, -)\simeq{\rm Hom}_\Lambda(\mathbb{B}\otimes_\Lambda \Lambda, -)={\rm Hom}_\Lambda(\mathbb{B}, -).$$
Recall the isomorphism  ${\bf i}\simeq {\rm Hom}_\Lambda(\mathbb{B}, -)$ from  Lemma~\ref{lem:injres}. Combining the two isomorphisms, we obtain the required assertion.
\end{proof}

\begin{rem}\label{rem:exp-res}
Let us describe the dg-injective resolution more explicitly. For any complex $Y$, there is an embedding
$$\eta_Y\colon Y\longrightarrow \mathcal{Y}(\Lambda, Y)$$
sending $y\in Y$ to $\eta_Y(y)\in \mathcal{Y}_0(\Lambda, Y)=\Hom(\Lambda, Y)$ given by $a\mapsto ay$. We observe the following commutative triangle in $C(\Lambda\mbox{-Mod})$.
\[\xymatrix{
Y\ar[rr]^-{\eta_Y} \ar[dr]_-{{\rm Hom}_\Lambda(\varepsilon, Y)} && \mathcal{Y}(\Lambda, Y) \ar[dl]^-{\alpha_{\Lambda, Y}}\\
& {\rm Hom}_\Lambda(\mathbb{B}, Y)
}\]
It follows from Lemma~\ref{lem:injres} that $\eta_Y$ is a quasi-isomorphism. In view of Remark~\ref{rem:c}, we infer that $\eta_Y$ is a dg-injective resolution of $Y$.
\end{rem}

\begin{rem}\label{rem:exp-res-bi}
Taking $Y=\Lambda$ in Remark~\ref{rem:exp-res}, we observe that $\mathcal{Y}(\Lambda, \Lambda)$ is a complex of $\Lambda$-$\Lambda$-bimodule, whose right $\Lambda$-module structure is induced by the one on the second entry. Then  the  embedding
$$\eta_\Lambda\colon \Lambda\longrightarrow \mathcal{Y}(\Lambda, \Lambda)$$
is a cochain map between  complexes of bimodules. As $\alpha_{\Lambda, \Lambda}$ is an isomorphism of complexes of bimodules, we infer from Remark~\ref{rem:homo} that $\eta_\Lambda$ is  a homotopy equivalence on the right side; it is a dg-injective resolution of $\Lambda$ on the left side, as shown in Remark~\ref{rem:exp-res}. We mention that $\eta_\Lambda\colon \Lambda \rightarrow \mathcal{Y}(\Lambda, \Lambda)^{\rm op}$ is a homomorphism of dg algebras, and thus a quasi-isomorphism of dg algebras. Here, ``op" means the opposite dg algebra.
\end{rem}

In view of Proposition~\ref{prop:Y-i}, the following result is indicated by \cite[Proposition~7.3]{CW}.

\begin{prop}\label{prop:Y-quasi}
For any complexes $X$ and $Y$ of $\Lambda$-modules, the following map
\[
\mathcal{Y}(X, Y)\longrightarrow \Hom_\Lambda(\mathcal{Y}(\Lambda, X), \mathcal{Y}(\Lambda, Y)), \quad f\longmapsto (g\mapsto f\odot g).
\]
is a quasi-isomorphism.
\end{prop}

\begin{proof}
Denote the above map by $\Psi$. Since $\mathcal{Y}(\Lambda, Y)$ is dg-injective,  the natural map induced by the quasi-isomorphism $\eta_X$ in Remark~\ref{rem:exp-res}
  $$\Hom_\Lambda(\mathcal{Y}(\Lambda, X), \mathcal{Y}(\Lambda, Y)) \longrightarrow  \Hom_\Lambda(X, \mathcal{Y}(\Lambda, Y))$$
is a quasi-isomorphism.  We have a sequence of isomorphisms of complexes
$$\Hom_\Lambda(X, \mathcal{Y}(\Lambda, Y))  \simeq \Hom_\Lambda(X, \Hom_\Lambda(\mathbb{B}, Y)) \simeq \Hom_\Lambda(\mathbb{B} \otimes_\Lambda X, Y)\simeq \mathcal{Y}(X, Y),$$
where the first and third isomorphisms use the isomorphism \eqref{iso:Y-B}, and the second one follows from the Hom-tensor adjunction.  Combining the above quasi-isomorphism with the composite isomorphism, we
 obtain an explicit quasi-isomorphism
 $$\Phi\colon \Hom_\Lambda(\mathcal{Y}(\Lambda, X), \mathcal{Y}(\Lambda, Y)) \longrightarrow   \mathcal{Y}(X, Y).$$
 We observe that $\Phi$ sends $\phi\colon \mathcal{Y}(\Lambda, X)\rightarrow \mathcal{Y}(\Lambda, Y)$ to an element in $\mathcal{Y}(X, Y)$, whose component in $\mathcal{Y}_p(X, Y)$ is described as follows: $$s\bar{a}_{1, p}\otimes x\longmapsto (-1)^{p\cdot |x|}(\phi\circ \eta_X)(x)(s\bar{a}_{1, p}\otimes 1_\Lambda).$$
Here, we abuse  $(\phi\circ \eta_X)(x)\in \mathcal{Y}(\Lambda, Y)$ with its component in $\mathcal{Y}_p(\Lambda, Y)$. Using (\ref{equation:yonedaproduct}), it is direct to verify that $\Phi\circ \Psi$ equals the identity map on $\mathcal{Y}(X, Y)$. Then we are done.
\end{proof}

\section{An explicit quasi-isomorphism}\label{sec:epsilon}

In this section, we study an explicit quasi-isomorphism $\epsilon_X$ and a related triangulated subcategory $\mathcal{K}$ of $\mathbf{K}(\Lambda\mbox{-Mod})$. The results  will be used in Section~\ref{sec:8}.

For each complex $X$, we consider the following explicit map between complexes of $\Lambda$-modules:
\begin{align}\label{equ:epsilon}
\epsilon_X\colon \mathcal{Y}(\Lambda, \Lambda)\otimes_\Lambda X\longrightarrow \mathcal{Y}(\Lambda, X).
\end{align}
For  any $f\in \mathcal{Y}_p(\Lambda, \Lambda)$ and $x\in X$,  the element $\epsilon_X(f\otimes_\Lambda x)\in \mathcal{Y}_p(\Lambda, X)$ is described as follows: it sends $s\bar{a}_{1, p}\otimes b\in s\bar{\Lambda}^{\otimes p}\otimes \Lambda$ to $f(s\bar{a}_{1, p}\otimes b)x\in X$. The map is essentially induced by the composition $\odot$ in $\mathcal{Y}$: for each $x$, we have $\eta_X(x)\in \mathcal{Y}_0(\Lambda, X)$ as in Remark~\ref{rem:exp-res}, and then
$$\epsilon_X(f\otimes_\Lambda x)= (-1)^{|x|\cdot |f|} \eta_X(x)\odot f.$$

We claim that $\epsilon_X$ is a quasi-isomorphism. Indeed, we observe the following commutative triangle in $C(\Lambda\mbox{-Mod})$.
\[\xymatrix{
\mathcal{Y}(\Lambda, \Lambda)\otimes_\Lambda X\ar[rr]^-{\epsilon_X} && \mathcal{Y}(\Lambda, X)\\
& \Lambda\otimes_\Lambda X = X  \ar[lu]^-{\eta_\Lambda \otimes_\Lambda {\rm Id}_X} \ar[ru]_-{\eta_X}
}\]
By Remark~\ref{rem:exp-res-bi}, $\eta_\Lambda$ is a homotopy equivalence on the right side. It follows that $\eta_\Lambda\otimes_\Lambda {\rm Id}_X$ is a quasi-isomorphism. Since $\eta_X$ is also a quasi-isomorphism, we infer the claim.

In general, the quasi-isomorphism $\epsilon_X$ may  not be a homotopy equivalence. Therefore, we consider the following full subcategory of $\mathbf{K}(\Lambda\mbox{-Mod})$
\begin{align}\label{defn:K}
\mathcal{K}:=\{X\in \mathbf{K}(\Lambda\mbox{-Mod})\; |\; \epsilon_X \mbox{ is a homotopy equivalence}\}.
\end{align}
We observe that $\mathcal{K}$ is a thick triangulated subcategory and that $\Lambda \in \mathcal{K}$.

\begin{lem}\label{lem:K}
A complex $X$ lies in $\mathcal{K}$ if and only if $\mathcal{Y}(\Lambda, \Lambda)\otimes_\Lambda X$ is $\mathbf{K}$-injective.
\end{lem}

\begin{proof}
The ``only if" part is clear, since $\mathcal{Y}(\Lambda, X)$ is dg-injective. The ``if" part holds, since any quasi-isomorphism between $\mathbf{K}$-injective complexes is necessarily a homotopy equivalence.
\end{proof}

In what follows, we assume that $\Lambda$ is left noetherian. Then coproducts of injective $\Lambda$-modules are still injective. Recall that a complex $X$ is called cohomologically bounded below, if $H^n(X)=0$ for $n\ll 0$.

\begin{prop}\label{prop:bounded-below}
Assume that $X$ is cohomologically bounded below. Then $\mathbb{B}_{\leq p}\otimes_\Lambda X$ belongs to $\mathcal{K}$ for any $p\geq 0$.
\end{prop}

\begin{proof}
We claim  $\mathcal{Y}(\Lambda, \Lambda)\otimes_\Lambda (\mathbb{B}^{-q}\otimes_\Lambda X)$ is $\mathbf{K}$-injective for each $q\geq 0$. Here, $\mathbb{B}^{-q} = \Lambda\otimes (s\bar{\Lambda})^{\otimes q}\otimes \Lambda$. Then by Lemma~\ref{lem:K}, each $\mathbb{B}^{-q}\otimes_\Lambda X$ belongs to $\mathcal{K}$. We observe that  $\mathbb{B}_{\leq p}\otimes_\Lambda X$ is an iterated extension of those complexes $\mathbb{B}^{-q}\otimes_\Lambda X$ in $\mathbf{K}(\Lambda\mbox{-Mod})$,  and recall that $\mathcal{K}$ is a triangulated subcategory of $\mathbf{K}(\Lambda\mbox{-Mod})$. Then we deduce the required statement.

We will actually prove that each $\mathcal{Y}(\Lambda, \Lambda)\otimes_\Lambda (\mathbb{B}^{-q}\otimes_\Lambda X)$ is dg-injective.  We have a canonical isomorphism of complexes
 $$\mathcal{Y}(\Lambda, \Lambda)\otimes_\Lambda (\mathbb{B}^{-q}\otimes_\Lambda X)\simeq \mathcal{Y}(\Lambda, \Lambda)\otimes ((s\bar{\Lambda})^{\otimes q}\otimes X).$$
Since $E$ is semisimple artianin and coproducts of injective $\Lambda$-modules are injective, we infer that $\mathcal{Y}(\Lambda, \Lambda)\otimes_\Lambda (\mathbb{B}^{-q}\otimes_\Lambda X)$ is a complex of injective $\Lambda$-modules.

We observe that as a complex of $E$-modules,  $V:= (s\bar{\Lambda})^{\otimes q}\otimes X$ is homotopy equivalent to $\bigoplus_{i\in \mathbb{Z}} \Sigma^{-i} H^i(V)$. Therefore, we deduce a homotopy equivalence
$$\mathcal{Y}(\Lambda, \Lambda)\otimes_\Lambda (\mathbb{B}^{-q}\otimes_\Lambda X)\simeq \bigoplus_{i\in \mathbb{Z}} \Sigma^{-i} \mathcal{Y}(\Lambda, \Lambda)\otimes  H^i(V).$$
Since $X$ is cohomologically bounded below, we infer that $H^i(V)=0$ for $i\ll 0$. This implies that the complex $\bigoplus_{i\in \mathbb{Z}} \Sigma^{-i} \mathcal{Y}(\Lambda, \Lambda)\otimes H^i(V)$ is bounded below. Then the claim follows, as any bounded below complex of injective modules is dg-injective.
\end{proof}

The proof of the following result is similar to the one in \cite[Proposition~5.2]{Ric}.

\begin{lem}\label{lem:generating}
Let $(P^n)_{n\in \mathbb{Z}}$ be a family of projective $\Lambda$-modules. Then the following canonical embedding of complexes of $\Lambda$-modules
$${\rm emb}\colon \bigoplus_{n\in \mathbb{Z}} \Sigma^{-n} \mathcal{Y}(\Lambda, P^n)\hookrightarrow \prod_{n\in \mathbb{Z}} \Sigma^{-n}\mathcal{Y}(\Lambda, P^n)$$
is a homotopy equivalence if and only if for each finitely generated $\Lambda$-module $M$ and $d\in \mathbb{Z}$,  the set $\{ n\in \mathbb{Z} \; |\; {\rm Ext}_\Lambda^{d-n}(M, P^n)\neq 0\}$ is finite.
\end{lem}

\begin{proof}
In this proof, we identify $\mathcal{Y}(\Lambda, -)$ with ${\bf i}$; see Proposition~\ref{prop:Y-i}.  We will identify ${\rm Ext}_\Lambda^{d-n}(M, P^n)$ with $H^d (\Sigma^{-n}{\rm Hom}_\Lambda({\bf i}(M), {\bf i}(P^n)))\simeq H^d ({\rm Hom}_\Lambda({\bf i}(M), \Sigma^{-n}{\bf i}(P^n)))$.  We will view ``emb" as  a morphism in $\mathbf{K}(\Lambda\mbox{-Inj})$.

Recall from \cite[Proposition~2.3]{Kra} that $\mathbf{K}(\Lambda\mbox{-Inj})$ is generated by ${\bf i}(M)$ for all finitely generated $\Lambda$-modules $M$. Therefore, the above embedding is a homotopy equivalence if and only if the following canonical map
 $${\rm Hom}_\Lambda ({\bf i}(M), \bigoplus_{n\in \mathbb{Z}} \Sigma^{-n}{\bf i}(P^n))\longrightarrow {\rm Hom}_\Lambda ({\bf i}(M),  \prod_{n\in \mathbb{Z}} \Sigma^{-n}{\bf i}(P^n))$$
 is a quasi-isomorphism for each $M$. Since each ${\bf i}(M)$ is compact, that is, the Hom functor ${\rm Hom}({\bf i}(M), -)$  commutes with infinite coproducts, the canonical embedding
 $$\bigoplus_{n\in \mathbb{Z}} {\rm Hom}_\Lambda ({\bf i}(M), \Sigma^{-n}{\bf i}(P^n))\longrightarrow {\rm Hom}_\Lambda ({\bf i}(M), \bigoplus_{n\in \mathbb{Z}} \Sigma^{-n}{\bf i}(P^n))$$
is always a quasi-isomorphism. We conclude that ``emb'' is a homotopy equivalence if and only if the following canonical map from a coproduct to  a product
 $$\bigoplus_{n\in \mathbb{Z}} {\rm Hom}_\Lambda ({\bf i}(M), \Sigma^{-n}{\bf i}(P^n))\longrightarrow \prod_{n\in \mathbb{Z}} {\rm Hom}_\Lambda ({\bf i}(M),  \Sigma^{-n}{\bf i}(P^n))$$
 is a quasi-isomorphism for each $M$. However, the latter condition is equivalent to the following finiteness one:  for each integer $d$, there are only finitely many $n$ with $H^d({\rm Hom}_\Lambda({\bf i}(M), \Sigma^{-n}{\bf i}(P^n))\neq 0$. Then we are done.
\end{proof}

We refer to  \eqref{defn:K} for the category $\mathcal{K}$.

\begin{prop}\label{prop:coprod}
Assume that $\Lambda$ is left noetherian. Then the following statements are equivalent.
\begin{enumerate}
\item Any complex $P$ of projective $\Lambda$-modules with zero differential belongs to $\mathcal{K}$.
\item The complex $\bigoplus_{n\in \mathbb{Z}} \Sigma^n(\Lambda)$ belongs to $\mathcal{K}$.
\item The complex $\bigoplus_{n\in \mathbb{Z}} \Sigma^n\mathcal{Y}(\Lambda, \Lambda)$ is dg-injective.
\item For each finitely generated $\Lambda$-module $M$, the set $\{ n\geq 0 \; |\; {\rm Ext}_\Lambda^n(M, \Lambda)\neq 0\}$ is finite.
\end{enumerate}
\end{prop}

\begin{proof}
 We observe the following isomorphism of complexes of injective modules
$$\bigoplus_{n\in \mathbb{Z}} \Sigma^n\mathcal{Y}(\Lambda, \Lambda)\simeq \mathcal{Y}(\Lambda, \Lambda)\otimes_\Lambda (\bigoplus_{n\in \mathbb{Z}} \Sigma^n(\Lambda)).$$
Then ``$(2) \Leftrightarrow (3)$'' follows from Lemma~\ref{lem:K}.

For ``$(3)\Leftrightarrow (4)$", we first observe that the canonical embedding
$$\bigoplus_{n\in \mathbb{Z}} \Sigma^n \mathcal{Y}(\Lambda, \Lambda)\hookrightarrow \prod_{n\in \mathbb{Z}} \Sigma^n \mathcal{Y}(\Lambda, \Lambda)$$
is a quasi-isomorphism, as $\mathcal{Y}(\Lambda, \Lambda)$ is quasi-isomorphic to $\Lambda$. Moreover, $\prod_{n\in \mathbb{Z}} \Sigma^n \mathcal{Y}(\Lambda, \Lambda)$ is a dg-injective. Therefore, the condition in (3) is equivalent to the one that the above embedding is a homotopy equivalence. Then the implications follow from Lemma~\ref{lem:generating}.

The implication ``$(1)\Rightarrow (2)$" is trivial. It remains to show ``$(4)\Rightarrow (1)$". The condition in (4) implies that for each finitely generated $\Lambda$-module $M$ and $d\in \mathbb{Z}$, the set
$$\{ n\in \mathbb{Z} \; |\; {\rm Ext}_\Lambda^{d-n}(M, P^n)\neq 0\}$$
 is finite. By the homotopy equivalence in Lemma~\ref{lem:generating}, we infer that the following infinite coproduct
$$\bigoplus_{n \in \mathbb{Z}} \Sigma^{-n}\mathcal{Y}(\Lambda, P^n)$$
is $\mathbf{K}$-injective. We observe a homotopy equivalence
$$\Sigma^{-n}\mathcal{Y}(\Lambda, P^n)\simeq \Sigma^{-n}\mathcal{Y}(\Lambda, \Lambda)\otimes_\Lambda P^n,$$
as both complexes are injective resolutions of the module $P^n$. We conclude that the following complex
$$\mathcal{Y}(\Lambda, \Lambda)\otimes_\Lambda P\simeq \bigoplus_{n\in \mathbb{Z}} \Sigma^{-n}\mathcal{Y}(\Lambda, \Lambda)\otimes_\Lambda P^n $$
is $\mathbf{K}$-injective. By Lemma~\ref{lem:K}, we infer (1).
\end{proof}

We mention that if the injective dimension of $\Lambda$ (i.e.\ the selfinjective dimension) on the left side, denote by ${\rm inj.dim}(_\Lambda\Lambda)$, is finite, then the equivalent conditions in Proposition~\ref{prop:coprod} hold. If $\Lambda$ is commutative, then these conditions are actually  equivalent to the condition that the localization of $\Lambda$ at any prime ideal has finite selfinjective dimension; see \cite[Theorem~I]{Goto}.  If $\Lambda$ is left artinian, these conditions also imply the finiteness of ${\rm inj.dim}(_\Lambda\Lambda)$.

Recall that $\Lambda$ is \emph{Gorenstein} provided that $\Lambda$ is two-sided noetherian and  has finite selfinjective dimension on each side. Therefore, the conditions in Proposition~\ref{prop:coprod} hold for any Gorenstein ring.

\begin{prop}\label{prop:Gorenstein}
Assume that $\Lambda$ satisfies the conditions in Proposition~\ref{prop:coprod}. Then for any complex $X$ and $p\geq 0$, the complex $\mathbb{B}_{\leq p}\otimes_\Lambda X$ belongs to $\mathcal{K}$.
\end{prop}

\begin{proof}
As in the first paragraph in the proof of Proposition~\ref{prop:bounded-below}, it suffices claim that each $\mathbb{B}^{-q}\otimes_\Lambda X$ belongs to $\mathcal{K}$.

By the isomorphism $\mathbb{B}^{-q}\otimes_\Lambda X\simeq \Lambda\otimes ((s\bar{\Lambda})^{\otimes q}\otimes X)$ and the semisimplicity of $E$, we infer that $\mathbb{B}^{-q}\otimes_\Lambda X$ is homotopy equivalent to a complex $P$ of projective $\Lambda$-modules with zero differential. Then the claim follows from Proposition~\ref{prop:coprod}(1).
\end{proof}

\section{Noncommutative differential forms}

In this section, we construct an explicit homotopy inverse $\iota_X$ in the Yoneda dg category $\mathcal Y$ of the dg-projection resolution $\varepsilon \otimes_\Lambda {\rm Id}_X\colon \mathbb B \otimes X \to X$, where the latter is viewed as an element in $\mathcal Y_0(\mathbb B \otimes X, X)$ via the inclusion \eqref{align:naturalinclusion}; see (\ref{equ:iotaX}). This homotopy inverse fits into the commutative diagram in Proposition~\ref{prop:Omega-bar}, which involves noncommutative differential forms and the truncated bar resolutions.

Let $X$ be a complex of $\Lambda$-modules. Following \cite[Section~8]{CW}, the complex of \emph{$X$-valued $E$-relative noncommutative differential $1$-forms} is defined by
$$\Omega_{{\rm nc}}(X)=s\bar{\Lambda}\otimes X,$$
 whose differential is given by $d(s\bar{a}\otimes x)=-s\bar{a}\otimes d_X(x)$, and whose grading is given such that $|s\bar{a}\otimes x|=|x|-1$.  The left  $\Lambda$-action is given by the following nontrivial rule:
\begin{align}\label{equ:nontri-rule}
b\blacktriangleright (s\bar{a}\otimes x)=s\overline{ba}\otimes x-s\bar{b}\otimes ax.
\end{align}
Indeed, $\Omega_{\rm nc}\colon \mathcal{Y}\rightarrow \mathcal{Y}$ is a dg endofunctor,  which sends a morphism $f\in \mathcal{Y}_p(X, Y)$ to the following morphism in $\mathcal{Y}_{p}(\Omega_{\rm nc}(X), \Omega_{\rm nc}(Y))$:
 $$   (s\bar{\Lambda})^{\otimes p}\otimes \Omega_{\rm nc}(X) = (s\bar{\Lambda})^{\otimes (p+1)}\otimes X \xrightarrow{\mathrm{Id}_{s\bar{\Lambda}}\otimes f}  s\bar{\Lambda}\otimes Y=\Omega_{\rm nc}(Y).$$
 We mention that the study of noncommutative differential forms goes back to \cite{CQ}.

\begin{lem}\label{lem:iso-Omega}
For each $p\geq 0$, we have a canonical  isomorphism
$$\mathbb{B}_{\geq p}\otimes_\Lambda \Omega_{\rm nc}(X)\simeq \mathbb{B}_{\geq p+1}\otimes_\Lambda X$$
of complexes of $\Lambda$-modules, sending $(a_0\otimes s\bar{a}_{1, n}\otimes 1)\otimes_\Lambda (s\bar{a}_{n+1}\otimes x)$ to $(a_0\otimes s\bar{a}_{1, n+1}\otimes 1)\otimes_\Lambda x$ for $n\geq p$.
\end{lem}

\begin{proof}
The given map is an isomorphism of graded $\Lambda$-modules. Using the external differential (\ref{equ:external}) and the nontrivial action (\ref{equ:nontri-rule}) on $\Omega_{\rm nc}(X)$, it is routine to verify that the isomorphism is compatible with differentials.
\end{proof}

 Following \cite[Section~8]{CW}, we have a closed natural transformation of degree zero
 $$\theta\colon {\rm Id}_\mathcal{Y}\longrightarrow \Omega_{\rm nc}$$
defined as follows: for any complex $X$, $\theta_X$ lies in $\mathcal{Y}_1(X, \Omega_{\rm nc}(X))\subseteq \mathcal{Y}(X, \Omega_{\rm nc}(X))$ and is given by
$$\theta_X(s\bar{a}\otimes x)=s\bar{a}\otimes x\in \Omega_{\rm nc}(X).$$

Recall from Lemma~\ref{lem:projres} the dg-projective resolution
$$\varepsilon\otimes {\rm Id}_X\colon \mathbb{B}\otimes_\Lambda X \longrightarrow X.$$
 It will be viewed as an element in $\mathcal{Y}_0(\mathbb{B}\otimes_\Lambda X, X)$ via \eqref{align:naturalinclusion}, and further as a closed morphism in $\mathcal{Y}$ of degree zero. Conversely, we will define another closed morphism in $\mathcal{Y}$ of degree zero
\begin{align}\label{equ:iotaX}
\iota_X\colon X\longrightarrow \mathbb{B}\otimes_\Lambda X.
\end{align}
For each $p\geq 0$, we define the entry $(\iota_X)_p\in \mathcal{Y}_p(X, \mathbb{B}\otimes_\Lambda X)$ by the following map:
$$(s\bar{\Lambda})^{\otimes p}\otimes X \longrightarrow \mathbb{B}^p\otimes_\Lambda X\subseteq \mathbb{B}\otimes_\Lambda X, \quad s\bar{a}_{1, p}\otimes x\longmapsto (1\otimes s\bar{a}_{1, p}\otimes 1)\otimes_\Lambda x.$$
Then we set $\iota_X=((\iota_X)_p)_{p\geq 0}\in \mathcal{Y}(X, \mathbb{B}\otimes_\Lambda X)$. The following identity holds in $\mathcal{Y}$:
\begin{align}\label{equ:iota-inv}
(\varepsilon\otimes_\Lambda {\rm Id}_X)\odot  \iota_X={\rm Id}_X.
\end{align}
Since $\varepsilon\otimes_\Lambda {\rm Id}_X$ is a quasi-isomorphism, it is a homotopy equivalence in $\mathcal{Y}$; see Proposition~\ref{cor:Y}(2). We infer from (\ref{equ:iota-inv})  that $\iota_X$ is a homotopy inverse of $\varepsilon\otimes_\Lambda {\rm Id}_X$.

\begin{rem}\label{rem:iota}
By Lemma~\ref{lem:Y-quasi-iso}, we infer that $\mathcal{Y}(\Lambda, \varepsilon\otimes_\Lambda {\rm Id}_X)$ is an isomorphism in $\mathbf{K}(\Lambda\mbox{-Mod})$. It follows from (\ref{equ:iota-inv}) that $\mathcal{Y}(\Lambda, \iota_X)$ is also an isomorphism in $\mathbf{K}(\Lambda\mbox{-Mod})$. Moreover, we have
$$\mathcal{Y}(\Lambda, \iota_X)^{-1}=\mathcal{Y}(\Lambda, \varepsilon\otimes_\Lambda {\rm Id}_X)$$
in $\mathbf{K}(\Lambda\mbox{-Mod})$.
\end{rem}

Denote by $\pi_0\colon \mathbb{B}\rightarrow \mathbb{B}_{\geq 1}=\mathbb{B}/ \mathbb{B}_{<1}$ the natural projection.

\begin{lem}\label{lem:comm-squ}
The following diagram
\[\xymatrix{
X \ar[r]^-{\theta_X} \ar[d]_-{\iota_X} & \Omega_{\rm nc}(X)\ar[d]^-{\iota_{\Omega_{\rm nc}(X)}}\\
\mathbb{B}\otimes_\Lambda X \ar[r] & \mathbb{B}\otimes_\Lambda \Omega_{\rm nc}(X)
}
\]
commutes in $\mathcal{Y}$, where the lower arrow is the composition of $\pi_0\otimes_\Lambda {\rm Id}_X$ with $\mathbb{B}_{\geq 1}\otimes_\Lambda X\rightarrow \mathbb{B}\otimes_\Lambda \Omega_{\rm nc}(X)$, the inverse of  the canonical isomorphism in Lemma~\ref{lem:iso-Omega}.
\end{lem}

\begin{proof}
Both the composite morphisms in the square correspond to the same element in $\mathcal{Y}(X, \mathbb{B}\otimes_\Lambda \Omega_{\rm nc}(X))$ given as follows: the entry in $\mathcal{Y}_0(X, \mathbb{B}\otimes_\Lambda \Omega_{\rm nc}(X))$ is zero, and the one in $\mathcal{Y}_p(X, \mathbb{B}\otimes_\Lambda \Omega_{\rm nc}(X))$ sends $s\bar{a}_{1, p}\otimes x$ to $(1\otimes s\bar{a}_{1, p-1}\otimes 1)\otimes_\Lambda (s\bar{a}_{p}\otimes x)$ for any $p\geq 1$.
\end{proof}

\begin{rem}
In view of (\ref{equ:iota-inv}) and in contrast to the above commutative diagram,  the following diagram in $\mathcal{Y}$
\[\xymatrix{
X \ar[r]^-{\theta_X}  & \Omega_{\rm nc}(X)\\
\mathbb{B}\otimes_\Lambda X  \ar[u]^-{\varepsilon\otimes_\Lambda {\rm Id}_X}\ar[r] & \mathbb{B}\otimes_\Lambda \Omega_{\rm nc}(X) \ar[u]_-{\varepsilon\otimes_\Lambda {\rm Id}_{\Omega_{\rm nc}(X)}}
}
\]
does {\it not} commute in general, as the two composite morphisms have differential filtration-degrees. This is one of the motivations to study the better-behaved morphisms $\iota_X$.
\end{rem}

By using Lemma~\ref{lem:iso-Omega} repeatedly, we obtain a canonical isomorphism
$$\varsigma_p\colon \mathbb{B}\otimes_\Lambda \Omega_{\rm nc}^p(X)\longrightarrow \mathbb{B}_{\geq p}\otimes_\Lambda X$$
of complexes of $\Lambda$-modules for each $p\geq 0$. Here, we have $\Omega_{\rm nc}^0(X)=X$ and $\mathbb{B}_{\geq 0}=\mathbb{B}$. Therefore, $\varsigma_0$ is the identity map. In more details, the isomorphism $\varsigma_p$ sends  $(a_0\otimes s\bar{a}_{1, n}\otimes 1)\otimes_\Lambda (s\bar{a}_{n+1,n+p}\otimes x)$ to $(a_0\otimes s\bar{a}_{1, n+p}\otimes 1)\otimes_\Lambda x$ for any $n\geq 0$.

Denote by   $\pi_p\colon \mathbb{B}_{\geq p}\rightarrow \mathbb{B}_{\geq p+1}$ the projection for any $p\geq 0$. The following commutative diagram is crucial in the proof of Theorem~\ref{thm:SY}.

\begin{prop}\label{prop:Omega-bar}
Keep the notation as above. Then for each $p\geq 0$, the following diagram
\[\xymatrix{
\Omega_{\rm nc}^p(X) \ar[rr]^-{\theta_{\Omega_{\rm nc}^p(X)}} \ar[d]_-{\varsigma_p\circ \iota_{\Omega_{\rm nc}^p(X)}} && \Omega_{\rm nc}^{p+1}(X)  \ar[d]^-{\varsigma_{p+1}\circ \iota_{\Omega_{\rm nc}^{p+1}(X)}}\\
\mathbb{B}_{\geq p}\otimes_\Lambda X \ar[rr]^-{\pi_p\otimes_\Lambda {\rm Id}_X} && \mathbb{B}_{\geq p+1}\otimes_\Lambda X
}
\]
commutes in $\mathcal{Y}$.
\end{prop}

\begin{proof}
By Lemma~\ref{lem:iso-Omega}, we have a canonical isomorphism
$$\psi\colon \mathbb{B}_{\geq 1}\otimes_\Lambda \Omega_{\rm nc}^p(X)\longrightarrow \mathbb{B}\otimes_\Lambda  \Omega_{\rm nc}^{p+1}(X).$$
 Applying Lemma~\ref{lem:comm-squ} to $\Omega_{\rm nc}^p(X)$, we obtain the following commutative diagram in $\mathcal{Y}$.
\[\xymatrix{
\Omega_{\rm nc}^p(X) \ar[rrr]^-{\theta_{\Omega_{\rm nc}^p(X)}} \ar[d]_-{\iota_{\Omega_{\rm nc}^p(X)}} && & \Omega_{\rm nc}^{p+1}(X)  \ar[d]^-{\iota_{\Omega_{\rm nc}^{p+1}(X)}}\\
\mathbb{B}\otimes_\Lambda \Omega_{\rm nc}^p(X) \ar[rrr]^-{\psi\circ (\pi_0\otimes_\Lambda {\rm Id}_{\Omega_{\rm nc}^p(X)})} &&& \mathbb{B}\otimes_\Lambda \Omega_{\rm nc}^{p+1}(X)
}
\]
We observe the following commutative diagram in $C_{\rm ac}(\Lambda\mbox{-Mod})$,
\[\xymatrix{
\mathbb{B}\otimes_\Lambda \Omega_{\rm nc}^p(X) \ar[d]_-{\varsigma_p} \ar[rrr]^-{\psi\circ (\pi_0\otimes_\Lambda {\rm Id}_{\Omega_{\rm nc}^p(X)})} &&& \mathbb{B}\otimes_\Lambda \Omega_{\rm nc}^{p+1}(X)\ar[d]^-{\varsigma_{p+1}}\\
\mathbb{B}_{\geq p}\otimes_\Lambda X \ar[rrr]^-{\pi_p\otimes_\Lambda {\rm Id}_X}
&&& \mathbb{B}_{\geq p+1}\otimes_\Lambda X
}
\]
which is also a commutative diagram in $\mathcal{Y}$. Combining the above two commutative squares, we obtain the required one.
\end{proof}

\section{The singular Yoneda dg category}\label{sec:SY}
In this section, we study the singular Yoneda dg category introduced in \cite{CW}, whose Hom complexes with the first entry $\Lambda$ will play a central role in this paper.

The \emph{$E$-relative singular Yoneda dg category} $\mathcal{SY}=\mathcal{SY}_{\Lambda/E}$ of $\Lambda$ is a dg category defined as follows: its objects are just complexes of $\Lambda$-modules; for two objects $X$ and $Y$, the Hom complex $\mathcal{SY}(X, Y)$ is defined to be the colimit of the following sequence of complexes.
$$\mathcal{Y}(X, Y)\longrightarrow \mathcal{Y}(X, \Omega_{\rm nc}(Y))\longrightarrow \cdots \longrightarrow \mathcal{Y}(X, \Omega_{\rm nc}^p(Y))\longrightarrow \mathcal{Y}(X, \Omega_{\rm nc}^{p+1}(Y))\longrightarrow \cdots$$
The structure map sends $f$ to $\theta_{\Omega_{\rm nc}^p(Y)}\odot f$. More precisely, for any $f \in  \mathcal{Y}_n(X, \Omega_{\rm nc}^p(Y))$, the map $ \theta_{\Omega_{\rm nc}^p(Y)}\odot f \in  \mathcal{Y}_{n+1}(X, \Omega_{\rm nc}^{p+1}(Y))$ is given by
\begin{align*}
s \bar{a}_{1, n+1} \otimes x \longmapsto (-1)^{|f|} s\bar{a}_1 \otimes f(s\bar{a}_{2, n+1}\otimes x).
\end{align*}

The image of $f\in \mathcal{Y}(X, \Omega_{\rm nc}^p(Y))$ in $\mathcal{SY}(X, Y)$ is denoted by $[f;p]$. The composition $\odot_{\rm sg}$ of $[f;p]$ with $[g;q]\in \mathcal{SY}(Y, Z)$ is defined by
\begin{align}\label{algin:odotsg}
[g;q]\odot_{\rm sg} [f;p]=[\Omega_{\rm nc}^p(g)\odot f;p+q].
\end{align}
We have the canonical dg functor $\mathcal{Y}\rightarrow \mathcal{SY}$, which acts on objects by the identity and sends $f$ to $[f;0]$.

We observe that for each complex $X$ of  $\Lambda$-modules, $\mathcal{SY}(\Lambda, X)$ is also a complex of  $\Lambda$-modules; its $\Lambda$-module structure  is induced from the right $\Lambda$-module structure on $\Lambda$.

\begin{lem}\label{lem:SY} Keep the notation as above.  Then the following two statements hold.
\begin{enumerate}
\item The stalk complex $\Lambda$ is contractible in $\mathcal{SY}$. In particular, the complex $\mathcal{SY}(\Lambda, X)$ is acyclic for any complex $X$.
    \item Any acyclic complex $X$ is contractible in $\mathcal{SY}$. Moreover, the complex $\mathcal{SY}(\Lambda, X)$ of $\Lambda$-modules is contractible for any acyclic complex $X$.
\end{enumerate}
\end{lem}

\begin{proof}
(1) We observe that $H^0(\mathcal{SY}(\Lambda, \Lambda))$ is isomorphic to ${\rm colim} \; H^0\mathcal{Y}(\Lambda, \Omega_{\rm nc}^p(\Lambda))$. As $\Omega_{\rm nc}^p(\Lambda)$ is a stalk complex concentrated on degree $-p$, it follows from (\ref{iso:Y-D}) that for each $p\geq 1$, $H^0\mathcal{Y}(\Lambda, \Omega_{\rm nc}^p(\Lambda))=0$. Then the required statements follow from Lemma~\ref{lem:contra}.

(2) By Proposition~\ref{cor:Y}(1), any acyclic complex $X$ is a contractible object in $\mathcal{Y}$. As any dg functor preserves contractible objects, it follows that $X$ is contractible in $\mathcal{SY}$.

Take $h\in \mathcal{Y}(X, X)$ of degree $-1$ with $\delta(h)={\rm Id}_X$, where $\delta$ denotes the differntial in $\mathcal{Y}$. For the contractibility of $\mathcal{SY}(\Lambda, X)$, we may choose the homotopy which sends any morphism $[f;p]$ to $[h; 0]\odot_{\rm sg}[f;p]$. Note that this homotopy is compatible with the left $\Lambda$-module structure on $\mathcal{SY}(\Lambda, X)$.
\end{proof}

By Lemma~\ref{lem:SY}(1), the following triangle functor
$$\mathcal{SY}(\Lambda, -)\colon \mathbf{K}(\Lambda\mbox{-Mod})\longrightarrow \mathbf{K}_{\rm ac}(\Lambda\mbox{-Mod})$$
is well defined. By  Lemma~\ref{lem:SY}(2) it vanishes on acyclic complexes, so we have the induced triangle functor
\begin{align*}\mathcal{SY}(\Lambda, -)\colon \mathbf{D}(\Lambda\mbox{-Mod})\longrightarrow \mathbf{K}_{\rm ac}(\Lambda\mbox{-Mod}).
\end{align*}

\begin{prop}\label{prop:SY}
Assume that $\Lambda$ is left noetherian. Then for each complex $X$ of $\Lambda$-modules, the complex $\mathcal{SY}(\Lambda, X)$ is acyclic and consists of injective $\Lambda$-modules.
\end{prop}

\begin{proof}
By Lemma~\ref{lem:SY}, the complex $\mathcal{SY}(\Lambda, X)$ is acyclic. It is the colimit of the following sequence
$$\mathcal{Y}(\Lambda, X)\longrightarrow \mathcal{Y}(\Lambda, \Omega_{\rm nc}(X))\longrightarrow \cdots \longrightarrow \mathcal{Y}(\Lambda, \Omega_{\rm nc}^p(X))\longrightarrow \mathcal{Y}(\Lambda, \Omega_{\rm nc}^{p+1}(X))\longrightarrow \cdots$$
whose each term is dg-injective; see Proposition~\ref{prop:Y-i}. Since $\Lambda$ is left noetherian, any direct limit of injective $\Lambda$-modules is injective. It follows that $\mathcal{SY}(\Lambda, X)$ is complex of injective $\Lambda$-modules.
\end{proof}

By the above proposition, we actually have a well-defined triangle functor
\begin{align}\label{fun:SY}
\mathcal{SY}(\Lambda, -)\colon \mathbf{D}(\Lambda\mbox{-Mod})\longrightarrow \mathbf{K}_{\rm ac}(\Lambda\mbox{-Inj}).
\end{align}

We will consider the following sequence of injective maps between complexes, which are induced by the inclusions $\mathbb{B}_{\leq p}\otimes_\Lambda X \subseteq \mathbb{B}_{\leq p+1}\otimes_\Lambda X$.
$$\mathcal{Y}(\Lambda, \mathbb{B}_{\leq 0}\otimes_\Lambda X)\hookrightarrow \mathcal{Y}(\Lambda, \mathbb{B}_{\leq 1}\otimes_\Lambda X)\hookrightarrow \mathcal{Y}(\Lambda, \mathbb{B}_{\leq 2}\otimes_\Lambda X)\hookrightarrow \cdots$$
We take the colimit, denoted by ${\rm colim}\; \mathcal{Y}(\Lambda, \mathbb{B}_{\leq p}\otimes_\Lambda X)$.  For each $p\geq 0$, we consider the following map
$$\mathcal{Y}(\Lambda, \mathbb{B}_{\leq p}\otimes_\Lambda X) \xrightarrow{\mathcal{Y}(\Lambda, {\rm inc})} \mathcal{Y}(\Lambda, \mathbb{B}\otimes_\Lambda X)\xrightarrow{\mathcal{Y}(\Lambda, \varepsilon\otimes_\Lambda {\rm Id}_X)} \mathcal{Y}(\Lambda, X),$$
where ``inc" denotes the inclusion $\mathbb{B}_{\leq p}\otimes_\Lambda X \subseteq \mathbb{B}\otimes_\Lambda X$. These maps are compatible with the ones in the sequence above, and induce the following one
\begin{align}\label{equ:vartheta}
\vartheta_X\colon {\rm colim}\; \mathcal{Y}(\Lambda, \mathbb{B}_{\leq p}\otimes_\Lambda X)\longrightarrow \mathcal{Y}(\Lambda, X).
\end{align}
More explicitly, $\vartheta_X$ sends an element represented by $f\in \mathcal{Y}_q(\Lambda, \mathbb{B}_{\leq q}\otimes_\Lambda X)$ to the composition $(\varepsilon\otimes_\Lambda {\rm Id}_X)\circ {\rm inc}\circ f\in \mathcal{Y}_q(\Lambda, X)$.

The following result shows that $\mathcal{SY}(\Lambda, X)$ is homotopy equivalent to the mapping cone ${\rm Cone}(\vartheta_X)$ of $\vartheta_X$. We denote by $C(\Lambda\mbox{-Inj})$ the category of complexes of injective $\Lambda$-modules.

\begin{thm}\label{thm:SY}
Assume that $\Lambda$ is left noetherian. Then for any complex $X$ of $\Lambda$-modules, we have an exact triangle in $\mathbf{K}(\Lambda\mbox{-}{\rm Inj})$:
$$  {\rm colim}\;\mathcal{Y}(\Lambda, \mathbb{B}_{\leq p}\otimes_\Lambda X)\stackrel{\vartheta_X}\longrightarrow \mathcal{Y}(\Lambda, X)\longrightarrow \mathcal{SY}(\Lambda, X)\longrightarrow \Sigma({\rm colim}\;\mathcal{Y}(\Lambda, \mathbb{B}_{\leq p}\otimes_\Lambda X)),$$
where the middle arrow is the canonical map, sending $f$ to $[f; 0]$.
\end{thm}

\begin{proof}
We apply $\mathcal{Y}(\Lambda, -)$ to the commutative diagram in Proposition~\ref{prop:Omega-bar}, and obtain the following commutative square in $C(\Lambda\mbox{-Inj})$.
\[\xymatrix{
\mathcal{Y}(\Lambda, \Omega_{\rm nc}^p(X)) \ar[rr]^-{\mathcal{Y}(\Lambda, \theta_{\Omega_{\rm nc}^p(X)})} \ar[d]_-{\mathcal{Y}(\Lambda, \varsigma_p\circ \iota_{\Omega_{\rm nc}^p(X)})} && \mathcal{Y}(\Lambda, \Omega_{\rm nc}^{p+1}(X))  \ar[d]^-{\mathcal{Y}(\Lambda, \varsigma_{p+1}\circ \iota_{\Omega_{\rm nc}^{p+1}(X)})}\\
\mathcal{Y}(\Lambda, \mathbb{B}_{\geq p}\otimes_\Lambda X) \ar[rr]^-{\mathcal{Y}(\Lambda, \pi_p\otimes_\Lambda {\rm Id}_X)} && \mathcal{Y}(\Lambda, \mathbb{B}_{\geq p+1}\otimes_\Lambda X)
}
\]
Recall that $\varsigma_p$ is an isomorphism. By Remark~\ref{rem:iota}, we infer that the vertical arrows are both homotopy equivalences. Taking the colimits along the horizontal maps, we obtain the following commutative diagram
\begin{equation}\label{comm:Y-Omega}
\xymatrix{
\mathcal{Y}(\Lambda, X) \ar[d]^-{\simeq}_-{\mathcal{Y}(\Lambda, \iota_X)} \ar[rr] &&\mathcal{SY}(\Lambda, X) \ar[d]_{\simeq}^-{{\rm colim}\; \mathcal{Y}(\Lambda, \varsigma_p\circ \iota_{\Omega_{\rm nc}^p(X)})}\\
\mathcal{Y}(\Lambda, \mathbb{B}\otimes_\Lambda X) \ar[rr] && {\rm colim}\; \mathcal{Y}(\Lambda, \mathbb{B}_{\geq p}\otimes_\Lambda X)
}
\end{equation}
By Lemma~\ref{lem:colimit-inj} below, the vertical arrow on the right side is a homotopy equivalence.

For each $p\geq 0$, we have an exact sequence of complexes as follows.
$$0\longrightarrow \mathcal{Y}(\Lambda, \mathbb{B}_{< p}\otimes_\Lambda X)  \xrightarrow{\mathcal{Y}(\Lambda, {\rm inc})}  \mathcal{Y}(\Lambda, \mathbb{B}\otimes_\Lambda X) \xrightarrow{\mathcal{Y}(\Lambda, {\rm pr})}  \mathcal{Y}(\Lambda, \mathbb{B}_{\geq p}\otimes_\Lambda X)\longrightarrow 0$$
Here, ``pr" denotes the projection. Letting $p$ vary and taking the colimits, we obtain an exact sequence of complexes of injective modules
$$0\longrightarrow {\rm colim}\; \mathcal{Y}(\Lambda, \mathbb{B}_{< p}\otimes_\Lambda X)\longrightarrow  \mathcal{Y}(\Lambda, \mathbb{B}\otimes_\Lambda X)  \longrightarrow {\rm colim}\; \mathcal{Y}(\Lambda, \mathbb{B}_{\geq p}\otimes_\Lambda X)\longrightarrow 0.$$
It induces an exact triangle in $\mathbf{K}(\Lambda\mbox{-}{\rm Inj})$:
$${\rm colim}\; \mathcal{Y}(\Lambda, \mathbb{B}_{< p}\otimes_\Lambda X)\rightarrow  \mathcal{Y}(\Lambda, \mathbb{B}\otimes_\Lambda X)  \rightarrow {\rm colim}\; \mathcal{Y}(\Lambda, \mathbb{B}_{\geq p}\otimes_\Lambda X)\rightarrow \Sigma({\rm colim}\; \mathcal{Y}(\Lambda, \mathbb{B}_{< p}\otimes_\Lambda X)).$$
We use the commutative diagram (\ref{comm:Y-Omega}) to replace the middle two terms in the above triangle, and obtain the required one. Here, we use
$${\rm colim}\; \mathcal{Y}(\Lambda, \mathbb{B}_{< p}\otimes_\Lambda X)={\rm colim}\; \mathcal{Y}(\Lambda, \mathbb{B}_{\leq p}\otimes_\Lambda X);$$
moreover,  we need the fact that in $\mathbf{K}(\Lambda\mbox{-Inj})$, we have $\mathcal{Y}(\Lambda, \iota_X)^{-1}= \mathcal{Y}(\Lambda, \varepsilon\otimes_\Lambda {\rm Id}_X)$; see Remark~\ref{rem:iota}.
\end{proof}

The following result is standard.

\begin{lem}\label{lem:colimit-inj}
Assume that $\Lambda$ is left noetherian. Suppose that we are given a commutative diagram in $C(\Lambda\mbox{-}{\rm Inj})$ with each $g_p$  a homotopy equivalence.
\[\xymatrix{
I_0 \ar[d]_-{g_0} \ar[r]^{\phi_0} & I_1  \ar[d]^-{g_1} \ar[r]^-{\phi_1} & I_2 \ar[d]^-{g_2}\ar[r] & \cdots \ar[r] & I_p  \ar[d]^-{g_p}\ar[r]^-{\phi_p} & I_{p+1} \ar[d]^-{g_{p+1}} \ar[r] & \cdots\\
J_0\ar[r]^{\psi_0} & J_1 \ar[r]^-{\psi_1} & J_2\ar[r] & \cdots \ar[r] & J_p \ar[r]^-{\psi_p} & J_{p+1}\ar[r] & \cdots
}\]
 Then  the induced map
$${\rm colim} \; g_p\colon {\rm colim} \; I_p\longrightarrow {\rm colim} \; J_p$$
is a homotopy equivalence.
\end{lem}

\begin{proof}
We observe that the following exact sequence of complexes
$$0\longrightarrow \bigoplus_{p\geq 0} I_p \stackrel{\mathbf{1}-\phi}\longrightarrow \bigoplus_{p\geq 0} I_p \longrightarrow {\rm colim} \; I_p\longrightarrow 0$$
is componentwise split, where $\mathbf{1}-\phi$ is the unique map whose restriction on $I_p$ is given by $\binom{\mathbf{1}}{-\phi_p}\colon I_p\rightarrow I_p\oplus I_{p+1}$. Here, we use the fact that  $\bigoplus_{p\geq 0} I_p$ lies in $C(\Lambda\mbox{-}{\rm Inj})$, as $\Lambda$ is left noetherian.
In particular,  ${\rm colim} \; I_p$ also lies in $C(\Lambda\mbox{-}{\rm Inj})$. So we have an induced exact triangle in $\mathbf{K}(\Lambda\mbox{-Inj})$, as shown in the upper row of the following commutative diagram; compare \cite[Section~3.4]{Kra22}. Similarly, we have the lower exact triangle.
\[
\xymatrix{
\bigoplus_{p\geq 0} I_p \ar[d]_-{\bigoplus_{p\geq 0}g_p} \ar[r]^-{\mathbf{1}-\phi} &  \bigoplus_{p\geq 0} I_p \ar[d]^-{\bigoplus_{p\geq 0}g_p} \ar[r] & {\rm colim} \; I_p  \ar[d]^-{{\rm colim}\; g_p}\ar[r] & \Sigma (\bigoplus_{p\geq 0} I_p) \ar[d]^-{\Sigma(\bigoplus_{p\geq 0}g_p)}\\
\bigoplus_{p\geq 0} J_p  \ar[r]^-{\mathbf{1}-\psi} &  \bigoplus_{p\geq 0} J_p \ar[r] & {\rm colim} \; J_p \ar[r] & \Sigma (\bigoplus_{p\geq 0} J_p)
}\]
Since $\bigoplus_{p\geq 0}g_p$ is an isomorphism in $\mathbf{K}(\Lambda\mbox{-Inj})$, it follows that ${\rm colim}\; g_p$ is also an isomorphism in $\mathbf{K}(\Lambda\mbox{-Inj})$, as required.
\end{proof}

\section{The stabilization functor}

In  this  section, we describe the stabilization functor \cite{Kra} via the mapping cone of an explicit quasi-isomorphism; see Theorem~\ref{thm:S}.  We will assume that $\Lambda$ is left noetherian.

Recall from \cite{Kra} the following recollement.
\begin{equation}\label{rec:Kra}
\xymatrix{
\mathbf{K}_{\rm ac}(\Lambda\mbox{-Inj})  \ar[rr]^-{\rm inc} &&  \ar@/_1.5pc/[ll]|{\bar{\bf a}}  \ar@/^1.5pc/[ll]|{{\bf a}'} \mathbf{K}(\Lambda\mbox{-Inj})  \ar[rr]^{\rm can} && \mathbf{D}(\Lambda\mbox{-Mod}) \ar@/_1.5pc/[ll]|{\bar{\bf p}}  \ar@/^1.5pc/[ll]|{\bf i}
}
\end{equation}
Here, the lower part is the restriction of the one in (\ref{rec:calrevised}), so we use the same notation and we have $\mathbf{a'}= {\rm Hom}_\Lambda({\rm Cone}(\varepsilon), -)$ and $\mathbf{i} = {\rm Hom}_\Lambda(\mathbb{B}, -)\simeq  \mathcal{Y}(\Lambda, -)$; see Proposition~\ref{prop:Y-i}.
The functors in the upper row are nontrivial.

The following definition is taken from \cite[Section~5]{Kra}.

\begin{defn}\label{defn:S}
The \emph{stabilization functor} of $\Lambda$ is defined to be the composition
$$\mathbb{S}=\bar{\bf a}{\bf i}\colon \mathbf{D}(\Lambda\mbox{-}{\rm Mod})\longrightarrow \mathbf{K}_{\rm ac}(\Lambda\mbox{-}{\rm Inj}).$$
\end{defn}
We mention that by \cite[1.4.6]{BBD} or \cite[Subsection~2.1]{CL}, $\mathbb{S}$ is isomorphic to the composition $\Sigma {\bf a}'\bar{\bf p}$. As pointed out in Introduction, $\mathbb{S}$ is a triangulated analogue of the gluing functor in the dg setting \cite[Subsections~2.2 and~4.2]{KL}; for  a related $\infty$-categorical consideration, we refer to  \cite{Lur, DJW}.

Recall from Remark~\ref{rem:exp-res-bi} the embedding $\eta_\Lambda\colon \Lambda \rightarrow \mathcal{Y}(\Lambda, \Lambda)$ of complexes of $\Lambda$-$\Lambda$-bimodules. For any complex $X$, we denote by ${\bf p}(X) = \mathbb B \otimes_\Lambda X$ its dg-projective resolution. As $\Lambda$ is left noetherian, it follows that the complex $\mathcal{Y}(\Lambda, \Lambda)\otimes_\Lambda {\bf p}(X)$ consists of injective $\Lambda$-modules since $\mathcal{Y}(\Lambda, \Lambda)$ is a complex of injective modules; see Remark \ref{rem:c}.

The following result describes the functor $\bar{\bf p}$ in (\ref{rec:Kra}) explicitly.

\begin{prop}\label{prop:bar-p}
There is a natural isomorphism $\bar{\bf p}(X)\simeq \mathcal{Y}(\Lambda, \Lambda)\otimes_\Lambda {\bf p}(X)$ in $\mathbf{K}(\Lambda\mbox{-}{\rm Inj})$ for any complex $X$ of $\Lambda$-modules.
\end{prop}

\begin{proof}
Take any complex $I$ of injective modules. We first observe that the following map of complexes of $\Lambda$-modules
$${\rm Hom}_\Lambda(\eta_\Lambda, I)\colon   {\rm Hom}_\Lambda(\mathcal{Y}(\Lambda, \Lambda), I)\longrightarrow {\rm Hom}_\Lambda(\Lambda, I)=I$$
is a quasi-isomorphism. Indeed, according to Lemma~\ref{lem:injres}, we  identify $\mathcal{Y}(\Lambda, \Lambda)$ with ${\bf i}(\Lambda)$. Then we apply \cite[Lemma~2.1]{Kra}.

 We have the following quasi-isomorphisms of complexes.
 \begin{align*}
 {\rm Hom}_\Lambda(\mathcal{Y}(\Lambda, \Lambda)\otimes_\Lambda {\bf p}(X), I)\simeq {\rm Hom}_\Lambda({\bf p}(X), {\rm Hom}_\Lambda(\mathcal{Y}(\Lambda, \Lambda), I)) \simeq {\rm Hom}_\Lambda({\bf p}(X), I)
 \end{align*}
 Here, the map on the right hand side is given by ${\rm Hom}_\Lambda({\bf p}(X), {\rm Hom}_\Lambda(\eta_\Lambda, I))$; it is indeed a quasi-isomorphism, since ${\bf p}(X)$ is dg-projective and ${\rm Hom}_\Lambda(\eta_\Lambda, I)$ is a quasi-isomorphism. Recall that ${\rm Hom}_\Lambda({\bf p}(X), I)
$ computes ${\rm Hom}_{\mathbf{D}(\Lambda\mbox{-}{\rm Mod})}(X, I)$.  Applying $H^0(-)$ to the composite quasi-isomorphism above, we prove that $\mathcal{Y}(\Lambda, \Lambda)\otimes_\Lambda {\bf p}-$ is left adjoint to the canonical functor ``${\rm can}$''.
\end{proof}

Although the functor $\bar{\bf p}$ is explicitly given, we generally  do not have an explicit description for the relevant counit of the adjoint pair $(\bar{\bf p}, {\rm can})$, as explained below.

\begin{rem}\label{rem:counit}
 By the proposition above, we identify $\bar{\bf p}$ with $\mathcal{Y}(\Lambda, \Lambda)\otimes_\Lambda {\bf p}-$. Take  any complex $I$ of injective modules. Denote by $\pi_I\colon {\bf p}(I)\rightarrow I$ the dg-projective resolution. As ${\rm Hom}_\Lambda(\eta_\Lambda, I)$ is a surjective quasi-isomorphism, there is a cochain map $\xi\colon {\bf p}(I)\rightarrow {\rm Hom}_\Lambda(\mathcal{Y}(\Lambda, \Lambda), I)$ such that ${\rm Hom}_\Lambda(\eta_\Lambda, I)\circ \xi=\pi_I$ in the category $C(\Lambda\mbox{-}{\rm Mod})$. By the Hom-tensor adjunction, $\xi$ corresponds to the counit
$$u_I\colon \mathcal{Y}(\Lambda, \Lambda)\otimes_\Lambda {\bf p}(I) \longrightarrow I.$$
As $\xi$ is not explicit, we can not describe the counit $u_I$ explicitly.

However, by chasing the diagram, one proves that the following triangle
\begin{equation}\label{comm-tri}
\xymatrix{
{\bf p}(I)=\Lambda\otimes_\Lambda {\bf p} (I)\ar[dr]_-{\pi_I} \ar[rr]^-{\eta_\Lambda\otimes_\Lambda {\rm Id}_{{\bf p} (I)}} && \mathcal{Y}(\Lambda, \Lambda)\otimes_\Lambda {\bf p}(I) \ar[dl]^-{u_I}\\
& I
}
\end{equation}
commutes in $C(\Lambda\mbox{-}{\rm Mod})$.  Since $\eta_\Lambda$ is a homotopy equivalence on the right side, $\eta_\Lambda\otimes_\Lambda {\rm Id}_{{\bf p} (I)} $ is a quasi-isomorphism. It follows that so is $u_I$.  We mention that if $I$ is dg-injective, the above commutative triangle determines $u_I$ up to homotopy. We just use the fact that ${\rm Hom}_\Lambda({\rm Cone}(\eta_\Lambda\otimes_\Lambda {\rm Id}_{{\bf p} (I)}), I)$ is acyclic,  since ${\rm Cone}(\eta_\Lambda\otimes_\Lambda {\rm Id}_{{\bf p} (I)})$ is acyclic.
\end{rem}

To calculate the stabilization functor $\mathbb{S}$, we need the following well-known fact; compare \cite[the second paragraph in the proof of Lemma~3.1]{Bon} and \cite[Proposition~3.2.8]{Kra22}.

\begin{rem}\label{rem:unique-S}
 In the recollement (\ref{rec:Kra}), we have a functorial exact triangle in $\mathbf{K}(\Lambda\mbox{-Inj})$
\begin{align}\label{tri:Kra}
\bar{\bf p}(I)\stackrel{u_I}\longrightarrow I \longrightarrow \bar{\bf a}(I) \longrightarrow \Sigma \bar{\bf p}(I),
\end{align}
where $I$ is any complex of injective modules and  $u_I$ is the counit in Remark~\ref{rem:counit}; this triangle is unique.  Take any complex $X$ of $\Lambda$-modules. Suppose that there exists an exact triangle in $\mathbf{K}(\Lambda\mbox{-Inj})$:
$$I_1\longrightarrow {\bf i}(X)\longrightarrow I_2\longrightarrow \Sigma(I_1),$$
with $I_1\in {\rm Im}(\bar{\bf p})$ and $I_2\in \mathbf{K}_{\rm ac}(\Lambda\mbox{-Inj})$. Then there are unique isomorphisms $g_1\colon I_1\rightarrow \bar{\bf p}{\bf i}(X)$ and $g_2\colon I_2\rightarrow \mathbb{S}(X)$ making the following diagram commute.
\[
\xymatrix{
I_1\ar[r] \ar[d]_-{g_1} & {\bf i}(X) \ar@{=}[d]  \ar[r] & I_2 \ar[d]_-{g_2} \ar[r] & \Sigma(I_1) \ar[d]^-{\Sigma(g_1)}\\
\bar{\bf p}{\bf i}(X) \ar[r]^-{u_{{\bf i}(X)}} & {\bf i}(X) \ar[r] & \mathbb{S}(X) \ar[r]& \Sigma \bar{\bf p}{\bf i}(X)
}\]
Here, the lower exact triangle is obtained by applying (\ref{tri:Kra}) to ${\bf i}(X)$.
\end{rem}

For each complex $X$ of $\Lambda$-modules, we will consider the following composition:
\begin{align}\label{equ:kappa}
\kappa_X\colon \mathcal{Y}(\Lambda, \Lambda)\otimes_\Lambda \mathbb{B}\otimes_\Lambda X\xrightarrow{{\rm Id}_{\mathcal{Y}(\Lambda, \Lambda)}\otimes_\Lambda (\varepsilon\otimes_\Lambda {\rm Id}_X)} \mathcal{Y}(\Lambda, \Lambda)\otimes_\Lambda X\xrightarrow{\epsilon_X} \mathcal{Y}(\Lambda, X),
\end{align}
where $\epsilon_X$ is given in (\ref{equ:epsilon}). For a typical element $f\otimes_\Lambda (a_0\otimes s\bar{a}_{1, q}\otimes 1)\otimes_\Lambda x\in \mathcal{Y}(\Lambda, \Lambda)\otimes_\Lambda \mathbb{B}\otimes_\Lambda X$ with $f\in \mathcal{Y}_p(\Lambda, \Lambda)$, we have
$$\kappa_X(f\otimes_\Lambda (a_0\otimes s\bar{a}_{1, q}\otimes 1)\otimes_\Lambda x) \in \mathcal{Y}_{p}(\Lambda, X),$$
which sends $s\bar{b}_{1, p}\otimes b\in (s\bar{\Lambda})^{\otimes p}\otimes \Lambda$ to $\delta_{q, 0}f(s\bar{b}_{1, p}\otimes b)a_0x\in X$. Here, $\delta_{q, 0}$ is the Kronecker symbol.

Recall from Remark~\ref{rem:exp-res-bi} that $\mathcal{Y}(\Lambda, \Lambda)$ is homotopy equivalent to $\Lambda$ on the right side,  and by Lemma~\ref{lem:projres}, $\varepsilon\otimes_\Lambda {\rm Id}_X$ is a quasi-isomorphism. It follows that ${\rm Id}_{\mathcal{Y}(\Lambda, \Lambda)}\otimes_\Lambda (\varepsilon\otimes_\Lambda {\rm Id}_X)$ is a quasi-isomorphism. Since $\epsilon_X$ is also a quasi-isomorphism, we infer that so is $\kappa_X$. We conclude that ${\rm Cone}(\kappa_X)$ is an acyclic complex of injective $\Lambda$-modules. Consequently, we have a well-defined dg functor
$${\rm Cone}(\kappa_{-})\colon C_{\rm dg}(\Lambda\mbox{-Mod})\longrightarrow C_{\rm dg, ac}(\Lambda\mbox{-Inj}),$$
where $C_{\rm dg, ac}(\Lambda\mbox{-Inj})$ denotes the dg category formed by acyclic complexes of injective modules. This  dg functor induces a  well-defined triangle functor
$${\rm Cone}(\kappa_{-})\colon \mathbf{K}(\Lambda\mbox{-Mod})\longrightarrow \mathbf{K}_{\rm ac}(\Lambda\mbox{-Inj}).$$

We claim that for each quasi-isomorphism $g\colon X\rightarrow X'$, the corresponding map ${\rm Cone}(\kappa_{X})\rightarrow {\rm Cone}(\kappa_{X'})$ is a homotopy equivalence. Indeed, this map fits into the following commutative diagram of exact triangles in $\mathbf{K}(\Lambda\mbox{-Inj})$.
\[\xymatrix{
\mathcal{Y}(\Lambda, \Lambda)\otimes_\Lambda \mathbb{B}\otimes_\Lambda X \ar[r]^-{\kappa_X} \ar[d]_-{{\rm Id}_{\mathcal{Y}(\Lambda, \Lambda)\otimes_\Lambda \mathbb{B}} \otimes_\Lambda g}  & \mathcal{Y}(\Lambda, X) \ar[d]^-{\mathcal{Y}(\Lambda, g)} \ar[r]& {\rm Cone}(\kappa_{X}) \ar[d]\ar[r] & \Sigma(\mathcal{Y}(\Lambda, \Lambda)\otimes_\Lambda \mathbb{B}\otimes_\Lambda X) \ar[d]\\
\mathcal{Y}(\Lambda, \Lambda)\otimes_\Lambda \mathbb{B}\otimes_\Lambda {X'} \ar[r]^-{\kappa_{X'}} & \mathcal{Y}(\Lambda, X') \ar[r] & {\rm Cone}(\kappa_{X'}) \ar[r] & \Sigma(\mathcal{Y}(\Lambda, \Lambda)\otimes_\Lambda \mathbb{B}\otimes_\Lambda X')
}\]
The two vertical arrows on the left side are homotopy equivalences; see Lemma~\ref{lem:Y-quasi-iso}. Then the claim follows.

By the above claim, we have the following induced triangle functor
$${\rm Cone}(\kappa_{-})\colon \mathbf{D}(\Lambda\mbox{-Mod})\longrightarrow \mathbf{K}_{\rm ac}(\Lambda\mbox{-Inj}).$$

\begin{thm}\label{thm:S}
Keep the notation as above. Then we have an isomorphism of triangle functors
$${\rm Cone}(\kappa_{-})\simeq \mathbb{S}.$$
\end{thm}

\begin{proof}
We consider the standard triangle
$$\mathcal{Y}(\Lambda, \Lambda)\otimes_\Lambda \mathbb{B}\otimes_\Lambda {X} \stackrel{\kappa_{X}} \longrightarrow \mathcal{Y}(\Lambda, X) \longrightarrow  {\rm Cone}(\kappa_{X})\longrightarrow \Sigma(\mathcal{Y}(\Lambda, \Lambda)\otimes_\Lambda \mathbb{B}\otimes_\Lambda X).$$
By Proposition~\ref{prop:bar-p}, we identify  $\mathcal{Y}(\Lambda, \Lambda)\otimes_\Lambda \mathbb{B}\otimes_\Lambda {X}$ with $\bar{\bf p}(X)$. In particular, it lies in ${\rm Im}(\bar{\bf p})$. As mentioned above, ${\rm Cone}(\kappa_{X})$ is an acyclic complex of injective modules. We apply the uniqueness of the functorial exact triangle in Remark~\ref{rem:unique-S} to obtain a unique isomorphism $g_X\colon {\rm Cone}(\kappa_X)\rightarrow \mathbb{S}(X)$. This uniqueness also implies that $g_X$ is functorial in $X$, as required.
\end{proof}

\begin{rem}
The recollement \eqref{rec:Kra} may be rewritten as follows.
\begin{equation}
\xymatrix@C=4pc{
\mathbf{K}_{\rm ac}(\Lambda\mbox{-Inj})  \ar[rr]^-{\rm inc} &&  \ar@/_1.5pc/[ll]|{\bar{\bf a}}  \ar@/^1.5pc/[ll]|{{\bf a}' = {\rm Hom}_\Lambda({\rm Cone}(\varepsilon), -)} \mathbf{K}(\Lambda\mbox{-Inj})  \ar[rr]^{\rm can} && \mathbf{D}(\Lambda\mbox{-Mod}) \ar@/_1.5pc/[ll]|{\bar{\bf p}=\mathcal{Y}(\Lambda, \Lambda)\otimes_\Lambda \mathbb B \otimes_\Lambda -}  \ar@/^1.5pc/[ll]|{{\bf i} = \mathcal{Y}(\Lambda, -)}
}
\end{equation}
As mentioned in Remark~\ref{rem:counit}, the counit $u_I$ is not explicitly given.  Therefore, it is difficult to describe the functor $\bar {\bf a}$.  In this sense, the description of $\mathbb{S}=\bar{\bf a}{\bf i}$ in Theorem~\ref{thm:S}  is nontrivial.

As mentioned before, by \cite[1.4.6]{BBD} we have $\mathbb{S}\simeq \Sigma {\bf a}'\bar{\bf p}$. Therefore,
$$\mathbb{S}(X) \simeq \Sigma \Hom_\Lambda({\rm Cone}(\varepsilon), \bar{\bf p}(X)). $$
Applying the functor $\Hom_\Lambda(-, \bar{\bf p}(X))$ to the standard triangle \eqref{tri:bar}, we obtain an exact triangle  in $\mathbf{K}(\Lambda\mbox{-Inj})$
\[
 \bar{\bf p}(X) \xrightarrow{\eta_{\bar{\bf p}(X)}}   \mathcal{Y}(\Lambda, \bar{\bf p}(X))  \longrightarrow \mathbb{S}(X) \longrightarrow \Sigma \bar{\bf p}(X),
\]
where we apply Remark~\ref{rem:exp-res} to $\bar{\bf p}(X)$. Moreover, we have the following commutative diagram
\[
\xymatrix{
 \bar{\bf p}(X) \ar@{=}[d] \ar[r]^-{\eta_{\bar{\bf p}(X)}} &  \mathcal{Y}(\Lambda, \bar{\bf p}(X))  \ar[r] \ar[d]&  \mathbb{S}(X) \ar[r] \ar@{.>}[d] & \Sigma \bar{\bf p}(X)\ar@{=}[d] \\
  \bar{\bf p}(X) \ar[r]^-{\kappa_X} &  \mathcal{Y}(\Lambda, X)  \ar[r] &  {\rm Cone}(\kappa_X) \ar[r] & \Sigma \bar{\bf p}(X),
}
\]
where from the left, the second vertical arrow is $(\eta_{\mathcal Y(\Lambda, X)})^{-1} \circ \mathcal{Y}(\Lambda, \kappa_X)$, which is an isomorphism in $\mathbf{K}(\Lambda\mbox{-Inj})$. This yields another proof of Theorem~\ref{thm:S}.
\end{rem}

Denote by $\mathbf{D}^b(\Lambda\mbox{-mod})$ the bounded derived category of finitely generated $\Lambda$-modules. We view the bounded homotopy category $\mathbf{K}^b(\Lambda\mbox{-proj})$ of finitely generated projective modules as a triangulated subcategory of $\mathbf{D}^b(\Lambda\mbox{-mod})$. The \emph{singularity category} \cite{Buc, Orl} of $\Lambda$ is the Verdier quotient category
$$\mathbf{D}_{\rm sg}(\Lambda)=\mathbf{D}^b(\Lambda\mbox{-mod})/{\mathbf{K}^b(\Lambda\mbox{-proj})}.$$
It has a canonical dg enhancement, as explained below. Denote by $\mathcal{D}=\mathbf{D}_{\rm dg}^b(\Lambda\mbox{-mod})$ the bounded dg derived category, and by $\mathcal{P}$ its full dg subcategory formed by bounded complexes of projective modules. Following \cite{Kel18}, the \emph{dg singularity category} of $\Lambda$ is the dg quotient category $\mathcal{S}=\mathcal{D}/\mathcal{P}$. Then the homotopy category $H^0(\mathcal{S})$ is identified with $\mathbf{D}_{\rm sg}(\Lambda)$. For details on dg quotient categories, we refer to \cite{Kel99, Dri, CC}.

\begin{rem}
Keep the notation as above. We have the inclusion functor ${\rm inc}\colon \mathcal{P}\rightarrow \mathcal{D}$ and the quotient functor $\pi\colon \mathcal{D}\rightarrow \mathcal{S}$. By \cite[Proposition~4.6(ii)]{Dri}, these dg functors induce a recollement of derived categories; see also \cite[Theorem~5.1.3]{CC}.
\begin{equation*}
\xymatrix@C=4pc{
\mathbf{D}(\mathcal{S}) \ar[rr]^-{\rm can} &&  \ar@/_1.5pc/[ll]|{-\otimes^{\mathbb{L}}_\mathcal{D}\mathcal{S}}  \ar@/^1.5pc/[ll]|{\mathbb{R}{\rm Hom}_\mathcal{D}(\mathcal{S}, -)}  \mathbf{D}(\mathcal{D})  \ar[rr]^{\rm res} && \mathbf{D}(\mathcal{P}) \ar@/_1.5pc/[ll]|{-\otimes^{\mathbb{L}}_\mathcal{P}\mathcal{D}}  \ar@/^1.5pc/[ll]|{\mathbb{R}{\rm Hom}_\mathcal{P}(\mathcal{D}, -)}}
\end{equation*}
Here, for any small dg category $\mathcal{C}$, we denote by $\mathcal{D}(\mathcal{C})$ the derived category of right dg $\mathcal{C}$-modules; ``res" means the restriction of dg $\mathcal{D}$-modules to $\mathcal{P}$, and ``can" sends any dg $\mathcal{S}$-module $M$ to the composition $M\pi$. By \cite[Appendix~A]{Kra}, the recollement  \eqref{rec:Kra} is isomorphic to the above one; compare \cite[Theorem~2.2]{CLW}.  In comparison, we emphasize that the categories in \eqref{rec:Kra} seems to be more accessible. The stabilization functor $\mathbb{S}=\overline{\bf a} {\bf i}$ is isomorphic to $\mathbb{R}{\rm Hom}_\mathcal{P}(\mathcal{D}, -)\otimes^\mathbb{L}_\mathcal{D}\mathcal{S}$. However, the latter seems to be  hard to deal with.
\end{rem}

\section{Comparing the two functors}\label{sec:8}

In this final section, we compare the two triangle functors: $\mathcal{SY}(\Lambda, -)$ and $\mathbb{S}$. Both are from $\mathbf{D}(\Lambda\mbox{-Mod})$ to $\mathbf{K}_{\rm ac}(\Lambda\mbox{-Inj})$; see (\ref{fun:SY}) and Definition~\ref{defn:S}.

\subsection{A natural transformation}

By Theorem~\ref{thm:S}, we will identify $\mathbb{S}$ with ${\rm Cone}(\kappa_{-})$. Therefore, for each complex $X$ of $\Lambda$-modules, we have a standard exact triangle (\ref{tri:standard}) in $\mathbf{K}(\Lambda\mbox{-Inj})$:
\begin{align}\label{tri:S}
\mathcal{Y}(\Lambda, \Lambda)\otimes_\Lambda \mathbb{B}\otimes_\Lambda X \stackrel{\kappa_X}\longrightarrow \mathcal{Y}(\Lambda, X) \stackrel{\binom{\mathbf{1}}{0}} \longrightarrow \mathbb{S}(X)\stackrel{(0\; \mathbf{1})}\longrightarrow \Sigma(\mathcal{Y}(\Lambda, \Lambda)\otimes_\Lambda \mathbb{B}\otimes_\Lambda X).
\end{align}
By Proposition~\ref{prop:bar-p}, we have $\mathcal{Y}(\Lambda, \Lambda)\otimes_\Lambda \mathbb{B}\otimes_\Lambda X \simeq \bar{\bf p}(X)$. By Proposition~\ref{prop:SY} the complex $\mathcal{SY}(\Lambda, X)$ is acyclic and consisting  of injective modules. So, the upper half of the recollement (\ref{rec:Kra}) implies that
 $${\rm Hom}_{\mathbf{K}(\Lambda\mbox{-}{\rm Inj})}(\mathcal{Y}(\Lambda, \Lambda)\otimes_\Lambda \mathbb{B}\otimes_\Lambda X, \Sigma^n\mathcal{SY}(\Lambda, X))=0$$
  for any $n\in \mathbb{Z}$. Consider the canonical map $\mathcal{Y}(\Lambda, X)\rightarrow \mathcal{SY}(\Lambda, X)$ sending $f$ to $[f;0]$. It follows that there is a unique morphism in $\mathbf{K}(\Lambda\mbox{-Inj})$
$$c_X\colon \mathbb{S}(X)\longrightarrow \mathcal{SY}(\Lambda, X)$$
such that its composition with $\binom{\mathbf{1}}{0}\colon \mathcal{Y}(\Lambda, X)\rightarrow \mathbb{S}(X)$ equals the canonical map. The uniqueness of $c_X$ implies that it is functorial in $X$. We will use the natural transformation $c$ to compare the two functors.

We will use the following natural isomorphism to identify these complexes.
\begin{align}\label{iso:identi}
{\rm colim}\; \mathcal{Y}(\Lambda, \Lambda)\otimes_\Lambda \mathbb{B}_{\leq p}\otimes_\Lambda X\simeq \mathcal{Y}(\Lambda, \Lambda)\otimes_\Lambda \mathbb{B}\otimes_\Lambda X.
\end{align}
Recall from (\ref{equ:epsilon}) the following quasi-isomorphism for each $p\geq 0$.
$$\epsilon_{\mathbb{B}_{\leq p}\otimes_\Lambda X}\colon \mathcal{Y}(\Lambda, \Lambda)\otimes_\Lambda \mathbb{B}_{\leq p}\otimes_\Lambda X\longrightarrow \mathcal{Y}(\Lambda,  \mathbb{B}_{\leq p}\otimes_\Lambda X)$$
These quasi-isomorphisms induce a quasi-isomorphism
$${\rm colim}\; \epsilon_{\mathbb{B}_{\leq p}\otimes_\Lambda X} \colon {\rm colim}\; \mathcal{Y}(\Lambda, \Lambda)\otimes_\Lambda \mathbb{B}_{\leq p}\otimes_\Lambda X\longrightarrow {\rm colim}\; \mathcal{Y}(\Lambda, \mathbb{B}_{\leq p}\otimes_\Lambda X).$$
Thanks to the identification (\ref{iso:identi}) and by abuse of notation, we have the following quasi-isomorphism
\begin{align}\label{iso:colim-epsilon}
{\rm colim}\; \epsilon_{\mathbb{B}_{\leq p}\otimes_\Lambda X} \colon  \mathcal{Y}(\Lambda, \Lambda)\otimes_\Lambda \mathbb{B} \otimes_\Lambda X \longrightarrow {\rm colim}\; \mathcal{Y}(\Lambda, \mathbb{B}_{\leq p}\otimes_\Lambda X).
\end{align}
Finally,  we observe a canonical isomorphism
\begin{align}\label{iso:cone-colim}
{\rm Cone}({\rm colim}\; \epsilon_{\mathbb{B}_{\leq p}\otimes_\Lambda X})\simeq {\rm colim}\; {\rm Cone}(\epsilon_{\mathbb{B}_{\leq p}\otimes_\Lambda X})),
\end{align}
as taking cones and taking colimits are compatible.

The following main result describes the mapping cone of $c_X$ in terms of an explicit colimit.

\begin{thm}\label{thm:comparison}
Assume that $\Lambda$ is left noetherian. Then for each complex $X$, there is an exact triangle in $\mathbf{K}_{\rm ac}(\Lambda\mbox{-}{\rm Inj})$.
$${\rm colim}\; {\rm Cone}(\epsilon_{\mathbb{B}_{\leq p}\otimes_\Lambda X}) \longrightarrow \mathbb{S}(X)\stackrel{c_X}\longrightarrow \mathcal{SY}(\Lambda, X)\longrightarrow \Sigma({\rm colim}\; {\rm Cone}(\epsilon_{\mathbb{B}_{\leq p}\otimes_\Lambda X}))$$
Consequently, $c_X$ is a homotopy equivalence if and only if ${\rm colim}\; {\rm Cone}(\epsilon_{\mathbb{B}_{\leq p}\otimes_\Lambda X})$ is contractible.
\end{thm}

\begin{proof}
We will compare (\ref{tri:S}) with the exact triangle in Theorem~\ref{thm:SY}. It is direct to verify that the following diagram
\[
\xymatrix{
\mathcal{Y}(\Lambda, \Lambda)\otimes_\Lambda \mathbb{B}\otimes_\Lambda X  \ar[d]_-{{\rm colim}\; \epsilon_{\mathbb{B}_{\leq p}\otimes_\Lambda X}} \ar[rr]^-{\kappa_X} && \mathcal{Y}(\Lambda, X)\ar@{=}[d]\\
{\rm colim}\; \mathcal{Y}(\Lambda, \mathbb{B}_{\leq p}\otimes_\Lambda X) \ar[rr]^-{\vartheta_X} && \mathcal{Y}(\Lambda, X)
}\]
commutes in $C(\Lambda\mbox{-Inj})$, where the cochain map ${\rm colim}\; \epsilon_{\mathbb{B}_{\leq p}\otimes_\Lambda X}$ is explained in (\ref{iso:colim-epsilon}). Therefore, we have a commutative diagram in $\mathbf{K}(\Lambda\mbox{-Inj})$:
\[
\xymatrix{
\mathcal{Y}(\Lambda, \Lambda)\otimes_\Lambda \mathbb{B}\otimes_\Lambda X  \ar[d]_-{{\rm colim}\; \epsilon_{\mathbb{B}_{\leq p}\otimes_\Lambda X}} \ar[r]^-{\kappa_X} & \mathcal{Y}(\Lambda, X)\ar@{=}[d] \ar[r] & \mathbb{S}(X) \ar@{.>}[d]\ar[r]& \Sigma(\mathcal{Y}(\Lambda, \Lambda)\otimes_\Lambda \mathbb{B}\otimes_\Lambda X)\ar[d] \\
{\rm colim}\; \mathcal{Y}(\Lambda, \mathbb{B}_{\leq p}\otimes_\Lambda X) \ar[r]^-{\vartheta_X} & \mathcal{Y}(\Lambda, X) \ar[r] & \mathcal{SY}(\Lambda, X)\ar[r] &  \Sigma({\rm colim}\; \mathcal{Y}(\Lambda, \mathbb{B}_{\leq p}\otimes_\Lambda X)).
}\]
By the above uniqueness of $c_X$, the dotted arrow has to be $c_X$. Applying the octahedral axiom (TR4) to the above diagram and using (\ref{iso:cone-colim}), we infer the required statement.
 \end{proof}

Now we  use the results in Section~\ref{sec:epsilon} to investigate when $c_X$ is a homotopy equivalence.

\begin{prop}\label{prop:comparison}
Keep the assumptions in Theorem~\ref{thm:comparison}. Then the following hold.
\begin{enumerate}
\item If $X$ is cohomologically bounded below, then $c_X$ is a homotopy equivalence. Consequently, the restriction of $c$ to $\mathbf{D}^{+}(\Lambda\mbox{-}{\rm Mod})$ is a natural isomorphism.
\item Assume that $\Lambda$ satisfies the equivalent conditions in Proposition~\ref{prop:coprod}. Then $c$ is a natural isomorphism.
\end{enumerate}
\end{prop}

\begin{proof}
We observe that for each $p\geq 0$, ${\rm Cone}(\epsilon_{\mathbb{B}_{\leq p}\otimes_\Lambda X})$ is a complex of injective $\Lambda$-modules. Therefore, by Lemma~\ref{lem:colimit-inj}, if each ${\rm Cone}(\epsilon_{\mathbb{B}_{\leq p}\otimes_\Lambda X})$ is contractible, or equivalently, the complex $\mathbb{B}_{\leq p}\otimes_\Lambda X$ lies in $\mathcal{K}$ defined in (\ref{defn:K}), then ${\rm colim}\; {\rm Cone}(\epsilon_{\mathbb{B}_{\leq p}\otimes_\Lambda X})$ is also contractible. By Theorem~\ref{thm:comparison}, $c_X$ is a homotopy equivalence. Now, we deduce (1) from Proposition~\ref{prop:bounded-below}, and (2) from Proposition~\ref{prop:Gorenstein}, respectively.
\end{proof}

\begin{rem}
Since Gorenstein rings satisfy the conditions in Proposition~\ref{prop:coprod}, the functors $\mathbb{S}$ and $\mathcal{SY}(\Lambda, -)$ are isomorphic if $\Lambda$ is Gorenstein. In the general case, we suspect that they are not isomorphic on $\mathbf{D}(\Lambda\mbox{\rm -Mod})$.
\end{rem}

\subsection{An application}

We will apply the above results and lift \cite[Corollary~5.4]{Kra} to the dg level; see Remark~\ref{rem:dg}.

\begin{cor}\label{cor:comparison}
Let $\Lambda$ be left noetherian, and $X$ be a cohomologically bounded below complex of $\Lambda$-modules. Then for any complex $Y$,  the following map
$$\Hom_\Lambda( \mathcal{SY}(\Lambda, X),  \mathcal{SY}(\Lambda, Y))  \longrightarrow \Hom_\Lambda(  \mathcal{Y}(\Lambda, X) ,  \mathcal{SY}(\Lambda, Y)), \quad \phi\mapsto \phi\circ q_X$$
is a quasi-isomorphism, where $q_X\colon \mathcal{Y}(\Lambda, X)\rightarrow \mathcal{SY}(\Lambda, X)$ is the canonical map.
\end{cor}

\begin{proof}
By the proof of Proposition~\ref{prop:comparison}, ${\rm colim}\; {\rm Cone}(\epsilon_{\mathbb{B}_{\leq p}\otimes_\Lambda X})$ is contractible, or equivalently, ${\rm colim}\; \epsilon_{\mathbb{B}_{\leq p}\otimes_\Lambda X}$ is a homotopy equivalence. It follows from the second commutative diagram in the proof of Theorem~\ref{thm:comparison}  that we have an exact triangle in $\mathbf{K}(\Lambda\mbox{-Inj})$
$$\bar{\bf p}(X)\longrightarrow \mathcal{Y}(\Lambda, X) \stackrel{q_X}\longrightarrow \mathcal{SY}(\Lambda, X)\longrightarrow \Sigma \bar{\bf p}(X).$$
Here, we identify $\bar{\bf p}(X)$ with $\mathcal{Y}(\Lambda, \Lambda)\otimes_\Lambda \mathbb{B}\otimes_\Lambda X$. By the adjoint pair $(\bar{\bf p}, {\rm can})$ in (\ref{rec:Kra}), we infer that ${\rm Hom}_\Lambda(\bar{\bf p}(X), \mathcal{SY}(\Lambda, Y))$ is acyclic. Applying ${\rm Hom}_\Lambda(-, \mathcal{SY}(\Lambda, Y))$ to the above exact triangle, we infer the required quasi-isomorphism.
\end{proof}

The following consideration is analogous to the one in Proposition~\ref{prop:Y-quasi}.  Recall from \eqref{algin:odotsg} the composition $\odot_{\sg}$ in the singular Yoneda dg category $\mathcal{SY}$. Then
we have the following map of complexes
\begin{align*}
\varphi_{X, Y} \colon \mathcal{SY}(X, Y)\longrightarrow \Hom_{\Lambda}( \mathcal{SY}(\Lambda, X), \mathcal{SY}(\Lambda, Y)), \quad [f;p] \longmapsto ([g;q] \mapsto [f;p] \odot_{\sg} [g;q]).
 \end{align*}
 If $X= Y$ then $\varphi_{X, X}$ is a homomorphism between dg endomorphism algebras.

  We identify $\mathbf{D}^b(\Lambda\mbox{\rm -mod})$ with the full subcategory of $\mathbf{D}(\Lambda\mbox{-Mod})$ formed by cohomologically bounded complexes with finitely generated cohomological modules.

\begin{prop}\label{prop:SY-quasi}
 Let $\Lambda$ be left noetherian and  $X\in \mathbf{D}^b(\Lambda\mbox{\rm -mod})$. Then the map $\varphi_{X, Y}$  is a quasi-isomorphism for any complex $Y$. Consequently,  $\varphi_{X,X}$ is a quasi-isomorphism of dg algebras for any $X\in \mathbf{D}^b(\Lambda\mbox{\rm -mod})$.
\end{prop}

\begin{proof}
We have the following natural maps between complexes.
\begin{align*}
\mathcal{SY}(X, Y) = {\rm colim}\; \mathcal{Y}(X, \Omega_{\rm nc}^p(Y))  &\longrightarrow  {\rm colim}\;  \Hom_\Lambda(\mathcal{Y}(\Lambda, X) , \mathcal{Y}(\Lambda, \Omega_{\rm nc}^p(Y)))\\
                       &\longrightarrow  \Hom_\Lambda(\mathcal{Y}(\Lambda, X) ,  \mathcal{SY}(\Lambda, Y))
\end{align*}
The first map is induced by the one in Proposition~\ref{prop:Y-quasi}, and thus a quasi-isomorphism. For the second one, we apply Lemma~\ref{lem:injres} to  identify $\mathcal{Y}(\Lambda, X)$ with ${\bf i}(X)$.  It follows from \cite[Proposition~2.3(2)]{Kra} that $\mathcal{Y}(\Lambda, X)$ is compact in $\mathbf{K}(\Lambda\mbox{-Inj})$. By \cite[Lemma~3.4.3]{Kra22}, we infer  that the second map is also a quasi-isomorphism.

Denote by $\psi_{X, Y}$ the quasi-isomorphism in Corollary~\ref{cor:comparison}.  It is routine to verify that the above composite quasi-isomorphism coincides with $\psi_{X, Y}\circ \varphi_{X, Y}$. This forces that $\varphi_{X, Y}$ is also a quasi-isomorphism.
\end{proof}

We recall that $\mathbf{D}_{\rm sg}(\Lambda)$ denotes the singularity category of $\Lambda$.

\begin{rem}\label{rem:dg}
Denote by $\mathcal{SY}^f$ the full dg subcategory of  $\mathcal{SY}$ formed by bounded complexes of finitely generated modules. Then $\mathcal{SY}^f$ is a dg enhancement of the singularity category $\mathbf{D}_{\rm sg}(\Lambda)$; see \cite[Corollary~9.3]{CW}. Proposition~\ref{prop:SY-quasi} implies that
the dg functor
$$\mathcal{SY}(\Lambda, -)\colon \mathcal{SY}^f\longrightarrow C_{\rm dg, ac}(\Lambda\mbox{-Inj})$$
is quasi-fully faithful. Taking their homotopy categories, we infer that
$$ \mathcal{SY}(\Lambda, -) \colon \mathbf{D}_{\rm sg}(\Lambda)\longrightarrow \mathbf{K}_{\rm ac}(\Lambda\mbox{-Inj})$$
is fully faithful; by abuse of notation, this functor is induced by  $ \mathcal{SY}(\Lambda, -) \colon \mathbf{D}^b(\Lambda\mbox{-mod})\rightarrow \mathbf{K}_{\rm ac}(\Lambda\mbox{-Inj}).$ This  recovers \cite[Corollary~5.4]{Kra}, as $\mathbb{S}$ and $\mathcal{SY}(\Lambda, -)$ are isomorphic on bounded complexes. As the discussion indicates, in comparison with $\mathbb{S}$, the triangle functor $\mathcal{SY}(\Lambda, -)$ naturally lifts to the dg level.
\end{rem}

We refer to \cite{ARS} for artin algebras and quivers.

\begin{rem}
Let $\Lambda$ be an artin algebra. Denote by $J=\mathrm{rad}(\Lambda)$ its Jacobson radical and set $E = \Lambda/J$. We assume that $E$ is a subalgebra of $\Lambda$ with $\Lambda=E\oplus J$. For instance, this assumption holds for any finite-dimensional algebra $\Lambda$ over a perfect field. Then
$\mathbb{S}(E) \simeq \mathcal{SY}(\Lambda, E)$ is a compact generator of $\mathbf{K}_{\rm ac}(\Lambda\mbox{-Inj})$. By Proposition~\ref{prop:SY-quasi}, the dg endomorphism algebra of  $\mathcal{SY}(\Lambda, E)$ is quasi-isomorphic to $\mathcal{SY}(E, E)$. By \cite[Theorem~9.5]{CW}, the latter is quasi-isomorphic to $L_E(J)^{\rm op}$, the opposite algebra of the \emph{dg Leavitt algebra} $L_E(J)$. 

In summary, the dg endomorphism algebra of $\mathbb{S}(E)$ is quasi-isomorphic to $L_E(J)^{\rm op}$, yielding another proof of \cite[Proposition~10.2]{CW}. If $\Lambda$ is given by a quiver with relations, the dg endomorphism algebra of $\mathbb{S}(E)$ is quasi-isomorphic to the opposite algebra of the dg Leavitt path algebra associated to the radical quiver of $\Lambda$; see \cite[Theorem 10.5]{CW}. 
\end{rem}

\vskip 5pt

\noindent {\bf Acknowledgements.}  \; The authors are very grateful to Bernhard Keller and Sergio Pavon for helpful comments. They thank  Srikanth Iyengar for the reference \cite{Goto}, and Gustavo Jasso for the reference \cite{KL}. This work is supported by the National Natural Science Foundation of China (No.s 11971449, 12131015, and 12161141001) and the Alexander von Humboldt Stiftung.

\vskip 10pt

 {\footnotesize \noindent Xiao-Wu Chen\\
 Key Laboratory of Wu Wen-Tsun Mathematics, Chinese Academy of Sciences\\
 School of Mathematical Sciences, University of Science and Technology of China\\
 Hefei 230026, Anhui, PR China\\
 xwchen$\symbol{64}$mail.ustc.edu.cn}

 \vskip 5pt

{\footnotesize \noindent Zhengfang Wang\\
Institute of Algebra and Number Theory, University of Stuttgart\\
Pfaffenwaldring 57, 70569 Stuttgart, Germany \\
 zhengfangw$\symbol{64}$gmail.com}


\begin{thebibliography}{9999}

\bibitem{ARS}{\sc M. Auslander, I. Reiten, and  S.O. Smal{\o}}, Representation Theory of Artin Algebras,  Cambridge Stud. Adv. Math. {\bf 36}, Cambridge Univ. Press,  Cambridge, 1995.


\bibitem{BBD} {\sc A.A. Beilinson, J. Bernstein, and P. Deligne}, {\em Faisceaux pervers}, Ast\'{e}risque {\bf 100}, Soc. Math. France, Paris, 1982.


\bibitem{Bon} {\sc A.I. Bondal}, {\em Representations of associative algebras and coherent sheaves}, Math. USSR Izv. {\bf 34}(1) (1990), 23--42.

\bibitem{BK} {\sc A.I. Bondal, and M.M. Kapranov}, {\em Enhanced triangulated categories}, Math. USSR Sb. {\bf 70}(1) (1991) 93--107.


\bibitem{Buc}{\sc R.-O. Buchweitz},  Maximal Cohen-Macaulay modules and Tate-cohomology over Gorenstein rings, with appendices by Luchezar L. Avramov, Benjamin Briggs, Srikanth B. Iyengar, and Janina C. Letz,  Math. Surveys and Monographs {\bf 262}, Amer. Math. Soc., 2021.

\bibitem{CC} {\sc X. Chen, and X.-W. Chen}, {\em An informal introduction to dg categories}, arXiv:1908.04599v2, 2019.

\bibitem{CL}{\sc X.-W. Chen, and J. Le}, {\em  Recollements, comma categories and morphic enhancements}, Proc. Roy. Soc.  Edinb. Sect. A: Math., 1--25, doi:10.1017/prm.2021.8.


\bibitem{CLW}{\sc X.-W. Chen, J. Liu, and R. Wang}, {\em Singular equivalences induced by bimodules and quadratic monomial algebras}, Algebr. Represent. Theor., 1--26,  doi:10.1007/s10468-021-10104-3.


\bibitem{CW}{\sc X.-W. Chen, and Z. Wang}, {\em The dg Leavitt algebra, singular Yoneda category and singularity category, with an appendix by Bernhard Keller and Yu Wang}, arXiv:2109.11278v2, 2021.


\bibitem{CQ} {\sc J. Cuntz, and D. Quillen}, {\em Algebra extensions and nonsingularity}, J. Amer. Math. Soc. {\bf 8}(2) (1995), 251--289.


\bibitem{Dri} {\sc V. Drinfeld}, {\em DG quotients of DG categories}, J. Algebra {\bf 272} (2) (2004), 643--691.


\bibitem{DJW} {\sc T. Dyckerhoff, G. Jasso, and T. Walde}, {\em Generalised BGP reflection functors via the Grothendieck construction}, Int. Math. Res. Not. IMRN {\bf 20} (2021), 15733--15745.


\bibitem{Goto} {\sc S. Goto}, {\em Vanishing of $Ext_A^i(M, A)$}, J. Math. Kyoto Univ. {\bf 22-3} (1982), 481--484.




\bibitem{Kel94} {\sc B. Keller}, {\em Deriving DG-categories}, Ann. Sci. \'Ecole Norm. Sup. {\bf 27}(4) (1994), 63--102.

\bibitem{Kel99} {\sc B. Keller}, \emph{On the cyclic homology of exact categories}, J. Pure Appl. Algebra {\bf 136} (1) (1999), 1--56.


\bibitem{Kel06} {\sc B. Keller}, {\em On differential graded categories}, International Congress of Mathematicians. Vol. II, 151--190, Eur. Math. Soc., Z\"urich, 2006.

\bibitem{Kel18}{\sc B. Keller}, {\em Singular Hochschild cohomology via the singularity category},
C. R. Math. Acad. Sci. Paris {\bf 356}  (11-12) (2018), 1106--1111. {\em Corrections}, C. R. Math. Acad. Sci. Paris {\bf 357} (6) (2019), 533--536. See also arXiv:1809.05121v10, 2020.


\bibitem{Kra}{\sc H. Krause}, {\em The stable derived category of a noetherian scheme}, Compos. Math. {\bf 141}(5) (2005), 1128--1162.

\bibitem{Kra22} {\sc H. Krause}, Homological Theory of Representations,  Cambridge Stud. Adv. Math. {\bf 195}, Cambridge Univ. Press, Cambridge, 2022.


\bibitem{KL} {\sc A. Kuznetsov, and V.A. Lunts}, {\em  Categorical resolutions of irrational singularities}, Int. Math. Res. Not. IMRN {\bf 13} (2015), 4536--4625.


\bibitem{Lur} {\sc J. Lurie}, Higher Algebra,  2017, available at: https://www.math.ias.edu/$\sim$lurie/papers/HA.pdf

\bibitem{Neeman} {\sc A. Neeman}, {\em The Grothendieck duality theorem via Bousfield's techniques and Brown representability}, J. Amer. Math. Soc. {\bf 9} (1996), 205--248.

\bibitem{Orl} {\sc D. Orlov,} {\em Triangulated categories of singularities and $D$-branes in Landau-Ginzburg models},  Proc. Steklov Inst. Math. {\bf 246}(3) (2004), 227--248.

\bibitem{Ric} {\sc J. Rickard}, {\em Unbounded derived categories and the finistic dimension conjecture}, Adv. Math. {\bf 354} (2019), 106735.


\bibitem{Wei}{\sc C.A. Weibel}, An Introduction to Homological Algebra, Cambridge Stud. Adv. Math. {\bf 38}, Cambridge Univ. Press, Cambridge, 1994.



\end{thebibliography}
\end{document}